\documentclass[11pt]{article}

\usepackage[T1]{fontenc}
\usepackage[width=15.2cm,height=21.5cm,centering]{geometry}
\usepackage{color}
\usepackage{graphicx}
\usepackage[small,bf]{caption}
\usepackage{subfig}
\usepackage{amsmath}
\usepackage{amssymb}
\usepackage{amsthm}
\usepackage[ec]{aeguill}
\usepackage{times}
\usepackage{bm,dsfont,stmaryrd}

\usepackage{mathrsfs}
\DeclareMathAlphabet{\mathscrbf}{OMS}{mdugm}{b}{d}

\usepackage{empheq}

\newtheorem{proposition}{Proposition}
\newtheorem{lemma}{Lemma}
\newtheorem{definition}{Definition}
\newtheorem{remark}{Remark}
\newtheorem*{remark*}{Remark}

\newcommand{\eps}{\varepsilon}
\newcommand{\beps}{\bm{\varepsilon}}
\newcommand{\epsz}{\eps_0}
\newcommand{\epse}{\eps_{\operatorname{eq}}}
\newcommand{\epszi}{{\eps_0^{(i)}}}
\newcommand{\epsei}{{\eps_{\operatorname{eq}}^{(i)}}}
\newcommand{\bepsei}{\bar{\eps}_{\operatorname{eq}}^{(i)}}
\newcommand{\bbeps}{\bar{\beps}}
\newcommand{\bepsz}{\bar{\eps}_0}
\newcommand{\bepse}{\bar{\eps}_{\operatorname{eq}}}

\newcommand{\nab}{\bm{\nabla}}
\newcommand{\bet}{\bm{\eta}}
\newcommand{\btau}{\bm{\tau}}
\newcommand{\bxi}{\bm{\xi}}
\newcommand{\bpsi}{\bm{\psi}}
\newcommand{\bK}{\bm{\mathcal{K}}}
\newcommand{\bJ}{\bm{\mathcal{J}}}
\newcommand{\bP}{\bm{\mathcal{P}}}

\newcommand{\Lm}{\bm{\mathcal{L}}}
\newcommand{\Lmz}{\bm{\mathcal{L}}_0}
\newcommand{\Lmh}{\,\widetilde{\!\bm{\mathcal{L}}}}
\newcommand{\sym}{_{\operatorname{s}}}
\newcommand{\sot}{\overset{\mbox{\tiny$2$}}{\otimes}\,}
\newcommand{\sots}{\overset{\mbox{\tiny$2$}}{\otimes}\sym\,}

\newcommand{\dc}{\!:\!}
\newcommand{\qc}{\!::\!}
\newcommand{\tr}{\operatorname{tr}}
\newcommand{\dev}{\bm{\mathrm{dev}}}

\newcommand{\bx}{\bm{x}}
\newcommand{\by}{\bm{y}}
\newcommand{\p}{\partial}

\newcommand{\bG}{\bm{\Gamma}}
\newcommand{\bGz}{\bm{\Gamma}_0}
\newcommand{\hbGz}{\hat{\bm{\Gamma}}_0}
\newcommand{\bGzi}{\bm{\Gamma}_0^\infty}

\newcommand{\td}{\,\text{d}}

\newcommand{\bLz}{\bm{\Lambda}_0}
\newcommand{\bPz}{\bm{\Pi}_0}
\newcommand{\bGziso}{\bm{\Gamma}_0^{\operatorname{iso}}}
\newcommand{\bGzperp}{\bm{\Gamma}_0^{\raisebox{0.5ex}{\tiny$\bot$}}}
\newcommand{\hbGzperp}{\hat{\bm{\Gamma}}_0^{\raisebox{0.5ex}{\tiny$\bot$}}}

\newcommand{\convd}{\,\underline{\underline{\ast}}\,}
\newcommand{\hdLm}{\widehat{\delta\Lm}}

\newcommand{\nI}{n_{\mbox{\tiny $I$ }}}
\newcommand{\nJ}{n_{\mbox{\tiny $J$ }}}
\newcommand{\nK}{n_{\mbox{\tiny $K$}}}
\newcommand{\nP}{n_{\mbox{\tiny $P$ }}}

\newcommand{\para}{_{\mbox{\tiny$\sslash$}}}
\newcommand{\orth}{_{\mbox{\tiny$\bot$}}}

\newcommand{\lJ}{\lambda_{\mbox{\tiny $J$ }}}
\newcommand{\lK}{\lambda_{\mbox{\tiny $K$ }}}
\newcommand{\lP}{\lambda_{\mbox{\tiny $P$ }}}

\title{Converting strain maps into elasticity maps\\ for materials with small contrast}

\author{C\'edric Bellis\\[4mm]
\small  Aix Marseille Univ, CNRS, Centrale Marseille, LMA, Marseille, France\\
\small \href{mailto:bellis@lma.cnrs-mrs.fr}{bellis@lma.cnrs-mrs.fr}
}

\begin{document}
\date{}
\maketitle


\begin{abstract}
This study addresses the question of the quantitative reconstruction of heterogeneous distributions of isotropic elastic moduli from full strain field data. This parameter identification problem exposes the need for a local reconstruction procedure that is investigated here in the case of materials with small contrast. To begin with the integral formulation framework for the periodic linear elasticity problem, first- and second-order asymptotics are retained for the strain field solution and the effective elasticity tensor. Properties of the featured Green's tensor are investigated to characterize its decomposition into an isotropic term and an orthogonal part. The former is then shown to define a local contribution to the volume integral equations considered. Based on this property, then the combination of multiple strain field solutions corresponding to well-chosen applied macroscopic strains is shown to lead to a set of local and uncoupled identities relating, respectively, the bulk and shear moduli to the spherical and deviatoric components of the strain fields. Valid at the first-order in the weak contrast limit, such relations permit point-wise conversions of strain maps into elasticity maps. Furthermore, it is also shown that for macroscopically isotropic material configurations a single strain field solution is actually sufficient to reconstruct either the bulk or the shear modulus distribution. Those results are then revisited in the case of bounded media. Finally, some sets of analytical and numerical examples are provided for comparison and to illustrate the relevance of the obtained strain-modulus local equations for a parameter identification method based on full-field data. 

\end{abstract}

\section{Introduction}

\subsection{Context}

Imaging the mechanical properties of a solid body is a problem with applications to material characterization, non-destructive testing or medical diagnosis. In this context, displacement or strain field measurements are generally assumed to be available within the domain or at its boundary and the reconstruction of distributions of elastic moduli from such data constitutes an ill-posed inverse problem. Over the last few decades a variety of identification methods have been developed for a number of constitutive models associated with a range of measurement modalities, see \cite{BonnetIP}.

The specific focus of the present study is the quantitative reconstruction of the elasticity maps that characterize a heterogeneous and linear isotropic elastic medium, from full strain field maps associated with a set of static mechanical excitations applied to the investigated body. This problem is motivated by the flourishing development of kinematic full-field measurement techniques, see e.g. \cite{Grediac:Hild:FF,book:3D:imaging,Parker}, of which digital image correlation is a representative example. In this context, the starting point for the present study is the availability of full strain field measurements from which we aim at inferring bulk and shear modulus distributions. To tackle this inverse problem a number of dedicated methods have been proposed and the reference \cite{Avril2008} proposes a comparative study. Overall, these approaches revolve around the interpretation of the momentum equation and of the constitutive law as equations for the unknown material parameters and where the spatially varying coefficients are constituted by the strain field data. In this respect, reconstructing the targeted physical parameters amounts to compute a solution to this partial differential equation owing to an appropriate integration strategy. Recently, so-called adjoint-weighted and gradient-based variational methods, see \cite{Barbone:2008,Barbone2010} and \cite{Bal:CPAM,BBIM}, have been developed to deal with intrinsic stability issues associated with this inverse problem and to improve the quality of reconstructions from noisy data.

Departing from these approaches, this article follows the alternate route of explicit reconstruction formulae. Their derivation is based on the integral formulation framework of the linear elasticity problem which allows to express a given strain field as the solution of a volume integral equation commonly referred to as the Lippmann-Schwinger equation. These formulations have been known and used for a long time, see e.g. \cite{Willis1981,Walpole}, in particular due to their interest for computing the behavior of composite materials and their relevance to homogenization theories. It is especially noteworthy that such formulations are at the core of efficient fast-Fourier transform based numerical methods for simulating the behavior of complex microstructured materials \cite{Moulinec94,Moulinec:98,Lebensohn} and they can been adapted to solve the inverse problem as in \cite{BM16}. These iterative approaches are based on the method of successive approximations for computing solutions to the integral equations considered. In this context, this study stems from evaluating strain field solutions at the first order with respect to fluctuations of the material parameters with respect to their mean values, an approximation which is also known in scattering theory as the Born approximation.

Therefore, this work aims at a direct reconstruction approach of bulk and shear modulus distributions from the knowledge of strain field solutions within the domain considered. Focusing on materials with small contrast then the governing integral equations are expressed in terms of the Green's operator for the periodic medium, to begin with, and expanded at the first-order in the material parameter perturbations. Upon characterizing the geometrical properties of the featured Green's tensor it is shown that specific experiments, associated with either purely hydrostatic or deviatoric applied macroscopic strains, can be combined to yield point-wise identities relating the sought moduli to the corresponding strain field solutions. In other words, explicit formulae are obtained at the first-order for a set of well-chosen experiments so that it makes it possible to reconstruct the unknown elastic moduli at a given point from strain field data at this point. This analysis is achieved by taking full advantage of the Fourier-based formulation, that is relevant to the periodic case, and of an orthogonal decomposition of the Green's tensor. Proceeding on this basis, the obtained results are extended to bounded domain configurations. Overall this local integration-free approach is analytic and it is thus based on noise-free strain field data. Note that the proposed method has a connection with works such as \cite{Bobeth,Kreher} on the quantification of local field fluctuations in heterogeneous materials. Lastly, the derivation of first-order inversion formulae is an objective also pursued in \cite{Ikehata} but in the different context of boundary measurements.

The outline of this article is as follows. The identification problem together with the governing equations are presented in the next subsection. The integral formulation framework of the periodic linear elasticity problem is investigated in Section \ref{sec:int:formulation} to shed light on key properties of the Green's operator. In this context, Section \ref{sec:strain:approach} focuses on the derivation of the sought strain-modulus identities for a generic material configuration. Section \ref{sec:energy:approach} addressed the special case of materials with macroscopic isotropy for which an alternate, but equivalent, energy-based approach is employed to obtain the associated reconstruction formulae. With this analysis at hand, the obtained results are then revisited in the case of bounded media in Section \ref{sec:bounded:dom}. Finally, a set of analytical and numerical examples are discussed in sections \ref{sec:analytical:ex} and \ref{sec:numerical:ex} to support and illustrate the proposed reconstruction method. 

\subsection{Statement of the problem}

Consider a periodic elastic medium with representative volume element $\mathcal{V}\subset\mathbb{R}^d$, with $d=2$ or $3$, and characterized by the elasticity tensor field $\Lm \in L^{\infty}_\mathrm{per}\big(\mathcal{V}, \sots \!(\!\sots\! \mathbb{R}^d)\big)$. The medium is considered to be such that the characteristic length-scale of the inhomogeneities is small compared to that of the domain $\mathcal{V}$. The static governing equations read
\begin{equation}\label{eq:elas}
\nab\times[\nab\times\beps(\bx)]^t=\bm{0}, \qquad \bm{\sigma}(\bx)=\Lm(\bx): \beps(\bx), \qquad \nab\cdot\bm{\sigma}(\bx)=\bm{0}, 
\end{equation}
with periodic strain field $\beps:\mathcal{V}\to\sots \mathbb{R}^d$ and associated stress field $\bm{\sigma}:\mathcal{V}\to\sots \mathbb{R}^d$ that satisfies anti-periodic boundary conditions on $\partial \mathcal{V}$. Furthermore, the mechanical loading is assumed to be compatible with a uniform macroscopic strain $\bbeps\in\sots \mathbb{R}^d$ such that
\begin{equation}\label{e:mean}
\langle \beps \rangle = \frac{1}{|\mathcal{V}|}\int_{\mathcal{V}}\beps(\bx)\,\td\bx=\bbeps.
\end{equation}
The strain field solution can be decomposed in terms of a $\mathcal{V}$-periodic displacement field $\bm{u}$ as
\[
\beps(\bx)=\bbeps+\nab\!\sym \bm{u}(\bx),
\]
where $\nab\!\sym$ denotes the symmetrized gradient operator.
Correspondingly, one defines the effective constitutive tensor ${\,\widetilde{\!\Lm}} \in\sots \!\big(\!\sots\! \mathbb{R}^d\big)$ as the fourth-order tensor that satisfies
\begin{equation}\label{def:L:eff}
\langle\bm{\sigma}\rangle={\,\widetilde{\!\Lm}}:\langle\bm{\varepsilon}\rangle.
\end{equation}

Introducing a reference elasticity tensor $\Lmz\in \sots \!\big(\!\sots\! \mathbb{R}^d\big)$ and noting $\delta\Lm=\Lm-\Lmz$, then owing to a perturbation approach the solution $\beps$ to the elasticity problem \eqref{eq:elas}--\eqref{e:mean} is known to satisfy the following Lippmann-Schwinger equation \cite{Willis1981,Milton}:
\begin{equation}\label{LS}
\beps(\bx)=\bbeps- \big[\bGz \, ( \delta\Lm : \beps)\big](\bx)
\end{equation}
where $\bGz$ is the periodic Green operator associated with $\Lmz$, such that $\beps'(\bx)=\big[\bGz \btau\big](\bx)$ is a compatible strain field solution in $\mathcal{V}$ satisfying $\langle \beps' \rangle=\bm{0}$ and $\nab\cdot\big(\Lmz:\beps'(\bx)+\btau(\bx)\big)=\bm{0}$.

By inverting Eqn. \eqref{LS} and using \eqref{def:L:eff}, then one obtains the following two relations which constitute the starting point for this work:
\begin{subequations}
\begin{empheq}[left=\empheqlbrace\,]{align}
& \beps(\bx)=\big(\bm{I}+\bGz\,  \delta\Lm \big)^{-1}:\bbeps, \label{key:eqns:1} \\[1mm]
& {\,\widetilde{\!\Lm}}= \Lmz+\big\langle \delta\Lm:\big(\bm{I}+\bGz\,  \delta\Lm \big)^{-1} \big\rangle. \label{key:eqns:2}
\end{empheq}
\end{subequations}

Focusing on an isotropic elastic material, the elasticity tensor $\Lm$ is expressed in terms of the bulk and shear modulus fields $\kappa$, $\mu \in L^{\infty}_\mathrm{per}(\mathcal{V}, \mathbb{R}^*_+)$ as
\begin{equation}\label{REFisot}
\Lm(\bx)=d \kappa(\bx) \bJ + 2\mu(\bx)\bK,
\end{equation}
according to the definitions of Appendix \ref{app:tensor}. Moreover, we choose the reference tensor $\Lmz$ to be isotropic as $\Lmz=d \kappa_0 \bJ + 2\mu_0 \bK$ where $\kappa_0$ and $\mu_0$ are assumed to be known a priori and to define a weak contrast configuration such that
\begin{equation}\label{def:elas:weak}
\Lm(\bx)=\Lmz+\delta\Lm(\bx) \quad \text{with } \|\delta\Lm\|=o\big(\|\Lmz\|\big).
\end{equation}
Accordingly, let introduce the notations $\delta\kappa(\bx)=\kappa(\bx)-\kappa_0$ and $\delta\mu(\bx)=\mu(\bx)-\mu_0$. Therefore, expanding the fundamental integral equations \eqref{key:eqns:1} and \eqref{key:eqns:2} at first and second orders respectively entails
\begin{subequations}\label{foa}
\begin{empheq}[left=\empheqlbrace\,]{align}
& \beps(\bx)=\bbeps - \big[\bGz \,( \delta\Lm : \bbeps)\big](\bx) + o\big(\|\delta\Lm\|\big), \label{key:eqns:1:exp}\\[1mm]
& {\,\widetilde{\!\Lm}}= \Lmz+\big\langle \delta\Lm \big\rangle - \big\langle \delta\Lm: \bGz\,  \delta\Lm  \big\rangle + o\big(\|\delta\Lm\|^2\big). \label{key:eqns:2:exp}
\end{empheq}
\end{subequations}
In Eqns. \eqref{def:elas:weak} and \eqref{foa}, the asymptotics are relative to the norm $\|\delta\Lm\|$. In the ensuing developments this term is to be interpreted as, see the analysis in \cite{Michel}: 
\[
\|\delta\Lm\|=\sup_{\bx\in\mathcal{V}}\, \max \big(|d\,\delta\kappa(\bx)| , |2\, \delta\mu(\bx)|\big).
\]

In the article we focus on materials with small contrasts for which we aim at deriving local formulae for expressing the elastic moduli in terms of strain field solutions. Such relations is intended to permit a point-wise identification of the elastic moduli $\kappa(\bx)$ and $\mu(\bx)$ from a set of local strain field measurements $\beps^{(i)}(\bx)$, with $i=1,\dots,N$, where $N$ to be determined denotes the number of experiments to perform. The availability of this internal data set is not intended to be discussed here but the reader is refered to the monographs \cite{Grediac:Hild:FF,book:3D:imaging} for a review of state-of-the-art experimental techniques constituting relevant non-invasive imaging modalities in solid mechanics. We rather focus hereafter on (i) the derivation of explicit local equations relating, in the weak contrast case, the elastic moduli to strain field data corresponding to full-field measurements, and (ii)~the characterization of the experiments to be performed.

\section{Integral formulation}\label{sec:int:formulation}

In this section, we investigate the integral equation \eqref{LS} which features the periodic Green's operator $\bGz$. The ensuing developments exploit some well-known properties of this operator, see e.g. \cite{Mura,Torquato1997}, that we establish hereafter for the reader's convenience.

\subsection{Green's operator}

Given the reference elasticity tensor $\Lmz$, the prescribed mean strain $\bbeps$ and a polarization term defined as $\btau(\bx)=\delta\Lm(\bx):\beps(\bx)$, let consider the following auxiliary problem in $\mathcal{V}$:
\begin{equation}\label{eq:ref:eps:tau}\left\{\begin{aligned}
& \bm{\sigma}(\bx)=\Lmz\dc\nab\!\sym \bm{u}(\bx)+\btau(\bx), && \nab\cdot\bm{\sigma}(\bx)=\bm{0},\\
 & \bm{u} \text{ is periodic}, && \bm{\sigma}\cdot\bm{n} \text{ is anti-periodic},\\
\end{aligned}\right.\end{equation}
where $\bm{n}$ denotes the unit outward normal vector on $\partial \mathcal{V}$. Using Appendix \ref{app:Fourier}, the equilibrium equation in \eqref{eq:ref:eps:tau} is recast in Fourier-space, as
\begin{equation}\label{equ:ref:u}
 \Big(\kappa_0+\frac{d-2}{d}\mu_0\Big)(\hat{\bm{u}}\cdot \bxi)\bxi + \mu_0  |\bxi|^2 \hat{\bm{u}}  = \frac{-1}{2\pi\mathrm{i}  }\,\hat\btau\cdot \bxi.
\end{equation}
Upon applying the inner product with $\bxi\neq\bm{0}$, the above equation entails
\begin{equation}\label{div:u}
\hat{\bm{u}}\cdot \bxi  = \frac{-1}{2\pi\mathrm{i}  |\bxi|^2} \, \frac{ d \,\bxi\cdot\hat\btau\cdot \bxi }{   d\kappa_0+2(d-1)\mu_0 }.
\end{equation}
Substitution of \eqref{div:u} in \eqref{equ:ref:u} yields for all $\bxi\neq\bm{0}$
\begin{equation}\label{eq:final:u}
2\pi\mathrm{i}  \,\hat{\bm{u}}(\bxi) = \frac{-1}{\mu_0}\,\frac{\hat\btau\cdot\bxi}{|\bxi|^2}+\frac{1}{\mu_0}\Big[ \frac{d\kappa_0+(d-2)\mu_0}{d\kappa_0+2(d-1)\mu_0} \Big]\frac{\bxi\cdot\hat\btau\cdot\bxi}{|\bxi|^4}\,\bxi .
\end{equation}
Finally, the strain field $\nab\!\sym \bm{u}$ solution to \eqref{eq:ref:eps:tau} is uniquely determined by
\begin{equation}\label{equ:eps:green}
\mathscr{F}\big[\nab\!\sym \bm{u}\big](\bm{0})=\bm{0} \quad \text{and} \quad \mathscr{F}\big[\nab\!\sym \bm{u}\big](\bxi)=-\hbGz(\bxi):\hat\btau(\bxi) \quad (\forall \bxi\neq\bm{0}),
\end{equation}
where the Green's tensor $\hbGz(\bxi)\in \sots \!\big(\!\sots\! \mathbb{R}^d\big)$ is a fourth-order tensor which is identified from Eqn. \eqref{eq:final:u} as
\begin{equation}\label{def:gamma:op}
\hbGz(\bm{0})=\bm{0} \quad \text{and} \quad \hbGz(\bxi)=\frac{\alpha_0}{|\bxi|^2} \bpsi_i(\bxi)\otimes\bpsi_i(\bxi) + \frac{\beta_0}{|\bxi|^4}\bpsi(\bxi)\otimes\bpsi(\bxi) \quad (\forall \bxi\neq\bm{0}),
\end{equation}
using the summation convention over repeated indices and where
\begin{equation}\label{def:psi}\left\{\begin{aligned}
& \bpsi(\bxi)=\bxi\otimes\bxi \quad \text{and} \quad \bpsi_i(\bxi)=\bxi\otimes\bm{e}_i+\bm{e}_i\otimes\bxi  \quad \text{for }i=1,\dots,d\\
& \alpha_0=\frac{1}{4\mu_0} \\
& \beta_0= - \frac{1}{\mu_0}\Big[ \frac{d\kappa_0+(d-2)\mu_0}{d\kappa_0+2(d-1)\mu_0}\Big].
\end{aligned}\right.\end{equation}
given an orthonormal basis $\{\bm{e}_i\}$ of $\mathbb{R}^d$. Therefore, according to the convolution theorem \eqref{conv:thm}, the operator $\bGz$ is defined as
\begin{equation}\label{def:op:Fourier}
\big[\bGz \btau\big](\bx)=\sum_{\substack{\bxi\in\mathcal{L}' \\ \bxi\neq\bm{0}} } \hbGz(\bxi):\hat\btau(\bxi) e^{2\pi \mathrm{i}\, \bx\cdot\bxi}.
\end{equation}
Upon substituting the polarization term $\btau(\bx)=\delta\Lm(\bx):\beps(\bx)$ in \eqref{equ:eps:green} one directly obtains the Lippmann-Schwinger equation \eqref{LS}.

By construction, the Green's tensor $\hbGz(\bxi)$ is a fourth-order tensor with minor and major index symmetries and which has transversely isotropic symmetry with axis $\bxi$. Based on this observation, the computation of its projections against the fourth-order isotropic tensors $\bJ$ and $\bK$ is the focus of the next section. 

\subsection{Orthogonal decomposition of the Green's tensor}

Key properties of the Green's operator $\bGz$ are now investigated.
\begin{lemma}\label{def:decomp:l:nl}
The isotropic component $\bGziso$ of the Green's tensor $\hbGz(\bxi)$ is independent of $\bxi$ for all $\bxi\in\mathbb{R}^d$. Therefore, one has the orthogonal decomposition 
 \[
 \hbGz(\bxi) =  \bGziso + \hbGzperp(\bxi), 
 \]
with these components satisfying the following properties:
\[
\displaystyle \bGziso=\frac{\lJ}{\nJ} \bJ + \frac{\lK}{\nK} \bK \quad\text{with}\quad  \left\{\begin{aligned} & \lJ= \hbGz(\bxi) \qc\bJ=\frac1d(4\alpha_0+\beta_0) \\ & \lK= \hbGz(\bxi) \qc\bK=\frac1d\Big[\big(2d(d+1)-4\big)\alpha_0+(d-1)\beta_0\Big] \end{aligned}\right.\]
 and $\hbGzperp(\bxi)\qc\bJ=\hbGzperp(\bxi)\qc\bK=0$.
\end{lemma}

\begin{proof} 
For all $\bxi\in\mathbb{R}^d$, the isotropic component $\bGziso$ of $\hbGz(\bxi)$ is found by computing the projections of the latter against the tensors $\bJ$ and $\bK$. First, using \eqref{comp:ela}, one has
\[
\hbGz(\bxi)\qc\bJ=\frac1d \Big[\frac{\alpha_0}{|\bxi|^2}\tr[\bpsi_i]\tr[\bpsi_i]+\frac{\beta_0}{|\bxi|^4}\tr[\bpsi]^2\Big].
\]
From the definition \eqref{def:psi}, the traces featured in the above expression are given by
\[
\tr[\bpsi_i]=2\xi_i \quad \text{for }i=1,\dots,d \quad \text{and} \quad \tr[\bpsi]=|\bxi|^2,
\]
so that, owing to the index summation convention, one obtains
\[
\lJ=\hbGz(\bxi)\qc\bJ=\frac1d(4\alpha_0+\beta_0).
\]
This relation is remarkably independent of the Fourier variable $\bxi$. Next, one has to compute the quantity $\hbGz(\bxi)\qc\bK$ which according to \eqref{comp:ela} reads
\[
\hbGz(\bxi)\qc\bK=\frac{\alpha_0}{|\bxi|^2}\Big[\bpsi_i\dc\bpsi_i-\frac1d \tr[\bpsi_i]\tr[\bpsi_i]\Big] + \frac{\beta_0}{|\bxi|^4}\Big[\bpsi\dc\bpsi-\frac1d \tr[\bpsi]^2\Big].
\]
On noting that
\[
\bpsi_i\dc\bpsi_i=2(|\bxi|^2+\xi_i^2) \quad \text{for }i=1,\dots,d \quad \text{and} \quad \bpsi\dc\bpsi=|\bxi|^4,
\]
one finally obtains after summation over index $i$
\[
\lK=\hbGz(\bxi)\qc\bK=\frac1d\Big[\big(2d(d+1)-4\big)\alpha_0+(d-1)\beta_0\Big].
\]
As for $\hbGz(\bxi)\qc\bJ$, the above quantity is constant in Fourier-space so that one can finally deduce that
 \[
\frac{1}{\nJ}\big(\hbGz(\bxi)\qc\bJ \big) \bJ + \frac{1}{\nK}\big(\hbGz(\bxi)\qc\bK \big) \bK =\bGziso \in \sots \!\big(\!\sots\! \mathbb{R}^d\big) \\
 \]
According to its definition, the tensor $\hbGzperp(\bxi)$ denotes the component of $\hbGz(\bxi)$ that is orthogonal to isotropic tensors, hence it satisfies $\hbGzperp(\bxi)\qc\bJ=\hbGzperp(\bxi)\qc\bK=0$.
  \end{proof}
 
According to the orthogonal decomposition of Lemma \ref{def:decomp:l:nl}, the Green's operator $\bGz$ can be decomposed into a local term that involves the tensor $\bGziso$ and a non-local operator $\bGzperp$ constructed from $\hbGzperp$ as in \eqref{def:op:Fourier}, i.e.
 \begin{equation}\label{def:ortho:decomp}
 \big[\bGz\btau\big](\bx)=\bGziso\dc\btau(\bx) + \big[\bGzperp \btau\big](\bx).
 \end{equation}
 This decomposition is the cornerstone for the approach investigated in this article.
 
 \begin{remark}\label{rmk:l:nl}
The identity \eqref{def:ortho:decomp} can be put into the broader perspective of determining what are the local and non-local contributions to the integral equation formulation of a given elasticity problem. On noting $\bG$ the Green's operator associated with a homogeneous elasticity tensor $\Lmz$ for the domain considered, then one would seek a decomposition such as
 \begin{equation}\label{def:decomp:l:nl:gen}
 \big[\bG\btau\big](\bx)=\bLz\dc\btau(\bx) + \big[\bPz\, \btau\big](\bx).
\end{equation}
where $\bLz$ and $\bPz$ denotes some local and non-local operators. That is the integration of the singular kernel of the Green's operator that leads to an identity such as \eqref{def:decomp:l:nl:gen}. For example, considering the case of an infinite medium, then the conventional integration approach, see \cite{Mikhlin,Bonnet:book}, yields
\[ \bLz=\bm{\mathcal{S}}_{\!B}:\Lmz^{-1} \qquad \text{and} \qquad
\big[\bPz\, \btau\big](\bx) = P.V.\int_{\mathbb{R}^d} \bGzi(\bx-\by):\btau(\by)\td\by,
\]
with $\bm{\mathcal{S}}_{\!B}$ being the Eshelby tensor for the ball that arises because of the exclusion of a spherical domain associated with the definition of Cauchy's principal value integral denotes as $P.V.\int_{\mathbb{R}^d}$ and where $\bGzi$ is the Green's function for the reference infinite medium considered.

The decompositions \eqref{def:ortho:decomp} and \eqref{def:decomp:l:nl:gen} actually coincide in the isotropic case, while $\bLz$ is no longer isotropic when $\Lmz$ is anisotropic. This implies that alternative orthogonal decompositions should be investigated in this case to obtain local identities that are analogous to the ones derived hereafter.
\end{remark}

\section{Strain-based approach}\label{sec:strain:approach}

In this section, we derive a set of local equations based on the strain field asymptotics \eqref{key:eqns:1:exp} and on Lemma \ref{def:decomp:l:nl}.

\subsection{Preliminary results}

Based on the definition \eqref{orth:decomp} of the parallel and orthogonal components of the strain field relatively to the macroscopic strain $\bbeps$, taking the inner product of the first-order asymptotics \eqref{key:eqns:1:exp} with $\bbeps$ immediately leads to the following result:

\begin{lemma}\label{lemma:orth:para}
The strain field $\beps(\bx)$ associated with the macroscopic strain $\bbeps$ satisfies for all $\bx\in\mathbb{R}^d$:
\[\begin{aligned}
\eps\para(\bx)=\|\bbeps\| - \frac{1}{\|\bbeps\|}\big[\bbeps:\bGz \,( \delta\Lm : \bbeps)\big](\bx) + o\big(\|\delta\Lm\|\big).
\end{aligned}\]
\end{lemma}
Therefore, the strain field $\beps(\bx)$ is locally collinear to $\bbeps$ at the first order and according to \eqref{pyth:eps} one has also $\eps\orth(\bx)= O\big(\|\delta\Lm\|\big)$. However, in the ensuing analysis, only the parallel component $\eps\para$ of the strain field plays a central role in the procedure of deriving the sought local identities. This is indeed because of the quadratic involvement of the mean field $\bbeps$ in the quantity $\bbeps:\bGz \,( \delta\Lm : \bbeps)$.\\

We now consider that the prescribed mean strain $\bbeps$ is either purely spherical or purely deviatoric, i.e. $\bbeps$ satisfies $\bP\dc\bbeps=\bbeps$ with either $\bP=\bJ$ or $\bP=\bK$. According to the equations \eqref{REFisot} and \eqref{comp:ela}, the definition of the Green's operator and the identity \eqref{conv:thm}, one has
\begin{equation}\label{def:proj:lin}
\big[\bbeps:\bGz \,( \delta\Lm : \bbeps)\big](\bx) = \mathscr{F}^{-1}\Big[\widehat{\delta p}(\bxi)\,  \hbGz(\bxi)\qc\big(\bbeps\otimes\bbeps\big) \,\Big](\bx),
\end{equation}
where $\widehat{\delta p}(\bxi)=\nP^{-1}\big(\hdLm(\bxi)\qc\bP\big)$ and $\nP=\nJ$ or $\nK$ respectively, the term $\widehat{\delta p}$ being thus the Fourier transform of the corresponding isotropic elasticity parameter $d\,\delta\kappa$ or $2\,\delta\mu$.

In Eqn. \eqref{def:proj:lin}, the right-hand side scalar quantity $\hbGz(\bxi)\qc(\bbeps\otimes\bbeps)$ is likely to depend on the variable $\bxi$, which would in turn make this term non-local in the physical space. In other words, for an arbitrarily chosen macroscopic strain $\bbeps$, the corresponding strain field $\beps(\bx)$ depends on the constitutive moduli in a non-local fashion. However, given the properties of the Green's tensor in Lemma~\ref{def:decomp:l:nl}, then multiple experiments can be combined in order to retain only the constant isotropic component $\bGziso$ of the Green's tensor $\hbGz(\bxi)$ so as to reduce \eqref{def:proj:lin} to a local equation relating strain fields and elastic moduli. A simple way to achieve this is to consider macroscopic strains that constitute an orthogonal eigentensor basis of the fourth-order projection tensor $\bP$, i.e. $\bbeps^{(i)}\in\sots\mathbb{R}^d$ with $i=1,\dots,\nP$ such that
\begin{equation}\label{basis:P}
\bP=\sum_{i=1}^{\nP} \frac{\bbeps^{(i)}\otimes\bbeps^{(i)}}{\|\bbeps^{(i)}\|^2} \quad \text{and} \quad \bbeps^{(i)}\dc\bbeps^{(j)}=\|\bbeps^{(i)}\|\, \|\bbeps^{(j)}\|\,  \delta_{ij}.
\end{equation}
Therefore, considering the strain field solutions $\beps^{(i)}$ that satisfy $\langle\beps^{(i)}\rangle=\bbeps^{(i)}$, then using Eqns. \eqref{def:proj:lin}, \eqref{basis:P} and Lemma \ref{lemma:orth:para} entails upon summation
\begin{equation}\label{eq:proj:comb}
\sum_{i=1}^{\nP} \frac{\eps\para^{(i)}(\bx)}{\|  \bbeps^{(i)} \|} = \nP -  \mathscr{F}^{-1}\Big[\widehat{\delta p}(\bxi)\,  \hbGz(\bxi)\qc\bP \,\Big](\bx) + o\big(\|\delta\Lm\|\big).
\end{equation}
Based on Lemma \ref{def:decomp:l:nl}, i.e. $\hbGz(\bxi)\qc\bP=\bGziso\qc\bP=\lP$ with $\lP=\lJ$ or $\lP=\lK$ then Eqn. \eqref{eq:proj:comb} yield the following result. 
\begin{lemma}\label{lemma:prop:proj}
Consider a set of macroscopic strains $\bbeps^{(i)}$ with $i=1,\dots,\nP$ such that Eqn. \eqref{basis:P} holds for $\bP=\bJ$ or $\bP=\bK$. Considering the elastic modulus defined as $\delta p(\bx)=\nP^{-1}\,\delta\Lm(\bx)\qc\bP$ then the associated strain field solutions $\beps^{(i)}(\bx)$ satisfy the local equation
\[
 \delta p(\bx)   = \lP^{-1}  \sum_{i=1}^{\nP} \bigg[ 1 - \frac{\eps\para^{(i)}(\bx)}{\|  \bbeps^{(i)} \|} \bigg] + o\big(\|\delta\Lm\|\big), 
\]
or equivalently
\[
 \delta p(\bx)   =  \lP^{-1} \sum_{i=1}^{\nP} \bigg[ 1 -  \bigg( \frac{\bP \qc \beps(\bx)\otimes\beps(\bx) }{\bP \qc \bbeps\otimes\bbeps }\Bigg)^{\!\!1/2} \  \bigg] + o\big(\|\delta\Lm\|\big).
\]
\end{lemma}
In the above lemma the second identity is an immediate consequence of Lemma \ref{lemma:orth:para} and Eqns. \eqref{pyth:eps}--\eqref{pyth:JK} according to which the strain field solution $\beps(\bx)$ satisfies at the first-order:
\begin{equation}\label{eq:epspara:P}
\eps\para(\bx)=\big(\bP \qc \beps(\bx)\otimes\beps(\bx)\big)^{1/2}.
\end{equation}
In the next two subsections, the identities of Lemma \ref{lemma:prop:proj} are made explicit for the cases of purely spherical and purely deviatoric macroscopic strains.

\subsection{Spherical macroscopic strain}

Since the tensor $\bJ$ is associated with a subspace of dimension $\nJ=1$, then only one strain field measurement $\beps(\bx)$ corresponding to a purely spherical macroscopic strain $\bbeps$ is sufficient to identify the bulk modulus $\kappa(\bx)$. Given that
\[
\frac{\eps\para(\bx)}{\|  \bbeps \|}=\frac{\big(\bJ\dc\beps(\bx)\big)\dc\big(\bJ\dc\bbeps\big)}{\|\bbeps\|^2}=\frac{\tr[\beps(\bx)]\tr[\bbeps]}{d\|\bbeps\|^2},
\]
and computing $\lJ^{-1}$ based on the definition \eqref{def:psi} and Lemma \ref{def:decomp:l:nl}, then Lemma \ref{lemma:prop:proj} yields the following result.

\begin{proposition}\label{prop:sph:rela}
Let $\bbeps$ denote a purely spherical macroscopic strain with associated strain field solution $\beps(\bx)$. The latter satisfies, at the first order, the following local equation:
\begin{equation}\label{expl:rela:J:1}
\kappa(\bx) =  \kappa_0+\frac{\lJ^{-1}}{d}\bigg[1- \frac{\tr[\beps(\bx)]\tr[\bbeps]}{d\|\bbeps\|^2}\bigg] + o\big(\|\delta\Lm\|\big),
\end{equation}
where
\[
\lJ^{-1}=d\kappa_0+2(d-1)\mu_0.
\]
The identity \eqref{expl:rela:J:1} can be recast in terms of the hydrostatic strain field as
\begin{equation}\label{expl:rela:J}
\kappa(\bx) =  \kappa_0+\frac{\lJ^{-1}}{d}\bigg[1- \frac{\epsz(\bx)}{\bepsz}\bigg] + o\big(\|\delta\Lm\|\big),
\end{equation}
while the equivalent strain is such that $\epse(\bx) = O\big(\|\delta\Lm\|\big)$.
\end{proposition}

\subsection{Deviatoric macroscopic strains}

When considering the case $\bP=\bK$ in Lemma \ref{lemma:prop:proj} one has to deal with a larger subspace of dimension $\nK=2$ or $5$, for $d=2$ or $3$ respectively. Therefore, the derivation of a local identity for the shear modulus $\mu(\bx)$ requires the use of a number $\nK$ of experiments with applied deviatoric macroscopic strains.

\begin{proposition}\label{prop:dev:rela}
Consider a set of purely deviatoric strains $\bbeps^{(i)}$, for $i=1,\dots,\nK$, such that
\[
\bK=\sum_{i=1}^{\nK} \frac{\bbeps^{(i)}\otimes\bbeps^{(i)}}{\|\bbeps^{(i)}\|^2} \quad \text{with} \quad \bbeps^{(i)}\dc\bbeps^{(j)}=\|\bbeps^{(i)}\|\, \|\bbeps^{(j)}\|\,  \delta_{ij}.
\]
Then, at the first order, the associated strain field solutions $\beps^{(i)}(\bx)$ satisfy the local identity 
\begin{equation}\label{expl:rela:K:1}
\mu(\bx) = \mu_0+\frac{\lK^{-1}}{2}\sum_{i=1}^{\nK}\bigg[1- \frac{\dev[\beps^{(i)}(\bx)]\dc\dev[\bbeps^{(i)}]}{\|\bbeps^{(i)}\|^2}\bigg] + o\big(\|\delta\Lm\|\big),
\end{equation}
with 
\[
\lK^{-1}=\frac{2\mu_0 \big( d\kappa_0+2(d-1)\mu_0 \big)}{d(d-1)(\kappa_0+2\mu_0)}.
\]
Equation \eqref{expl:rela:K:1} can be recast in terms of equivalent strains as
\begin{equation}\label{expl:rela:K}
\mu(\bx)  = \mu_0+\frac{\lK^{-1}}{2}\sum_{i=1}^{\nK}\bigg[1-\frac{\epsei(\bx)}{\bepsei}\bigg] + o\big(\|\delta\Lm\|\big),
\end{equation}
while the hydrostatic strain fields satisfy $\epszi(\bx) = O\big(\|\delta\Lm\|\big)$.

\end{proposition}

\begin{remark}
In the local reconstruction formulae of propositions \ref{prop:sph:rela} and \ref{prop:dev:rela} it is assumed that the mean moduli $\kappa_0$, $\mu_0$ are known a priori. However, it can be seen that the only knowledge of the ratio $r_0=\kappa_0/\mu_0$ is actually sufficient to determine the quantities $\delta\kappa(\bx)/\kappa_0$ and $\delta\mu(\bx)/\mu_0$ uniquely.
\end{remark}

\section{Macroscopic isotropy and energy-based approach}\label{sec:energy:approach}

In this section we revisit the local equations \eqref{expl:rela:J} and \eqref{expl:rela:K} for the case of macroscopically isotropic materials. For such configurations, we establish that only one strain field is actually sufficient to reconstruct each elastic modulus locally. Although this result can be established using the strain-based approach of Section \ref{sec:strain:approach}, we propose here an alternative energy-based procedure. The overall approach revolves around the quantification of local strain field fluctuations through the derivation of second-order moments of the strain fields which are defined as mean values of strain-based quadratic quantities, see \cite{Bobeth,Kreher}.

 \subsection{Macroscopically isotropic configurations}

First, we state the definition of a macroscopically isotropic material as follows.
\begin{definition}\label{def:macro:isot}
A given medium is said to be macroscopically isotropic if the corresponding effective elasticity tensor ${\,\widetilde{\!\Lm}}$ defined by Eqn. \eqref{def:L:eff} is isotropic.
\end{definition}
Now, the aim is to characterize the quadratic quantity $\big\langle \delta\Lm: \bGz\,  \delta\Lm  \big\rangle$ for such configurations. According to the second-order asymptotics \eqref{key:eqns:2:exp}, this term satisfies   
\[
\big\langle \delta\Lm: \bGz\,  \delta\Lm  \big\rangle = \Lmz - {\,\widetilde{\!\Lm}} + \big\langle \delta\Lm \big\rangle + o\big(\|\delta\Lm\|^2\big).
\]
As a consequence, if the medium considered is both microscopically and macroscopically isotropic then the fourth-order tensor $\big\langle \delta\Lm: \bGz\,  \delta\Lm  \big\rangle$ is itself isotropic and there exist $a,b\in\mathbb{R}$ such that
\[
\big\langle \delta\Lm: \bGz\,  \delta\Lm  \big\rangle=a\,\bJ+b\,\bK.
\]
The coefficients $a$ and $b$ can be computed by projection against the tensors $\bJ$ and $\bK$. According to the identity \eqref{def:norm} and the definition of the Green's operator then one has for $\bP=\bJ$ or $\bP=\bK$:
\begin{equation}\label{macro:isot}\begin{aligned}
\big\langle \delta\Lm: \bGz\,  \delta\Lm  \big\rangle \qc \bP & = \sum_{\substack{\bxi\in\mathcal{L}' \\ \bxi\neq\bm{0}} } \big[\widehat{\delta\Lm}(\bxi):\hbGz(\bxi):\widehat{\delta\Lm}(\bxi)^*\big]\qc\bP  \\
& = \sum_{\substack{\bxi\in\mathcal{L}' \\ \bxi\neq\bm{0}} } \big(\widehat{\delta\mathcal{L}}(\bxi)\big)_{ijpq}\big(\hat{\Gamma}_0(\bxi)\big)_{pqrs}\big(\widehat{\delta\mathcal{L}}(\bxi)^*\big)_{rsk\ell}\mathcal{P}_{ijk\ell} \\
& = \sum_{\substack{\bxi\in\mathcal{L}' \\ \bxi\neq\bm{0}} } |\widehat{\delta p}(\bxi)|^2 \,  \hbGz(\bxi)\qc \bP
\end{aligned}\end{equation}
where $\widehat{\delta p}(\bxi)=\nP^{-1}\big(\hdLm(\bxi)\qc\bP\big)$ and the notation $|\!\cdot\!|$ denotes the complex modulus. Therefore, according to Lemma \ref{def:decomp:l:nl}, one obtains
\begin{equation}\label{macro:isot:2}
\big\langle \delta\Lm: \bGz\,  \delta\Lm  \big\rangle = \sum_{\substack{\bxi\in\mathcal{L}' \\ \bxi\neq\bm{0}} } \Big\{ |d\,\widehat{\delta \kappa}(\bxi)|^2 \frac{\lJ}{\nJ} \bJ + |2\,\widehat{\delta \mu}(\bxi)|^2 \frac{\lK}{\nK} \bK  \Big\}.
\end{equation}
Based on the definition of the tensor $\bGziso$ and by retracing the derivation of equations \eqref{macro:isot} and \eqref{macro:isot:2} one finally arrives at the following result.
\begin{lemma}\label{lemma:macro:isto}
If the material configuration considered is macroscopically isotropic in the sense of Definition \ref{def:macro:isot} then one has
\[
\big\langle \delta\Lm: \bGz\  \delta\Lm  \big\rangle =  \big\langle \delta\Lm: \bGziso :  \delta\Lm  \big\rangle.
\]
\end{lemma}

\subsection{Local strain field fluctuations}

In order to establish the sought local equations then the first step is to relate the local strain field fluctuations to the homogenized elasticity parameters. This can be achieved by quantifying the effective modulus sensitivities to a local parameter change. Owing to Hill's lemma, the effective elasticity tensor $\Lmh$ defined by Eqn. \eqref{def:L:eff} is such that the macroscopic and the averaged microscopic elastic energy densities are equal, i.e.
\begin{equation}\label{def:Lmh}
\big\langle\beps(\bx)\dc\Lm(\bx)\dc\beps(\bx)\big\rangle=\bbeps\dc\Lmh\dc\bbeps.
\end{equation}
Differentiating Eqn. \eqref{def:Lmh} with respect to a scalar parameter $t$ entails
\[
\big\langle\beps\dc\frac{\p\Lm}{\p t}\dc\beps\big\rangle + 2\big\langle\beps\dc\Lm\dc\frac{\p\beps}{\p t}\big\rangle=\frac{\p}{\p t}\big(\bbeps\dc\Lmh\dc\bbeps\big).
\]
In this equation the second left-hand side term vanishes owing to Hill's lemma, see e.g. \cite{PCS:NLC}, so that we finally obtain
\begin{equation}\label{derv:modulus}
\big\langle\beps\dc\frac{\p\Lm}{\p t}\dc\beps\big\rangle=\bbeps\dc\frac{\p\Lmh}{\p t}\dc\bbeps.
\end{equation}
Moreover, owing to the second-order asymptotics \eqref{key:eqns:2:exp} then one has
\begin{equation}\label{derv:Lmh}
\frac{\p\Lmh}{\p t}= \frac{\p}{\p t} \Big[ \big\langle \delta\Lm \big\rangle - \big\langle \delta\Lm: \bGz\,  \delta\Lm  \big\rangle  + o\big(\|\delta\Lm\|^2\big) \Big].
\end{equation}
Considering now that the medium is macroscopically isotropic, then substituting the identity of Lemma \ref{lemma:macro:isto} in the equations \eqref{derv:Lmh} and \eqref{derv:modulus} entails
\begin{equation}\label{eq:fluct}
\big\langle\beps\dc\frac{\p\Lm}{\p t}\dc\beps\big\rangle=\bbeps\dc\Big[ \big\langle \frac{\p}{\p t} \delta\Lm \big\rangle - 2\big\langle \delta\Lm: \bGziso :  \frac{\p}{\p t} \delta\Lm  \big\rangle  + \frac{\p}{\p t} o\big(\|\delta\Lm\|^2\big) \Big]:\bbeps.
\end{equation}
For the sake of simplicity it can be assumed that the representative volume element $\mathcal{V}$ considered is described by a piecewise-homogenous parameter distribution as:
\begin{equation}
\Lm(\bx)=\sum_{\phi=1}^{\Phi} \chi_\phi(\bx)\Lm_\phi,
\label{Ela:comp}
\end{equation}
where each phase $\phi$ is characterized by the indicator function $\chi_\phi(\bx)$ and the homogeneous elasticity tensor $\Lm_\phi$ with constant moduli $\kappa_\phi$, $\mu_\phi$. Moreover, we denote the corresponding elasticity contrast as $\delta\Lm_\phi=\Lm_\phi-\Lmz$. Note that the total number $\Phi\in\mathbb{N}$ of phases can be arbitrarily large in the ensuing developments. Following Definition \eqref{e:mean} of the spatial average over $\mathcal{V}$, one introduces the local average in a given phase $\phi$ occupying the domain $V_\phi\subset V$, as
\[
{\langle\eta\rangle}_\phi=\frac{1}{|V_\phi|}\int_{V_\phi}\eta(\bx)\td\bx   \\
\]
while $f_\phi=|V_\phi|/|V|$ denotes the corresponding volume fraction.

For such configurations, upon choosing $t=\nP^{-1}\,\delta\Lm_\phi\qc\bP=\delta p_\phi$ for a given phase $\phi$ with $\bP=\bJ$ or $\bP=\bK$ then $\frac{\p}{\p t} \delta\Lm=\chi_\phi(\bx)\bP$ so that Eqn. \eqref{eq:fluct} leads to the following result.

\begin{lemma}\label{lemma:energy}
If the medium considered is macroscopically isotropic then for any given phase $\phi$:
\[
{\langle\beps\dc \bP \dc\beps\rangle}_\phi = \big( 1 - 2\,\lP\nP^{-1}\,\delta p_\phi\big) \, \bbeps\dc\bP\dc\bbeps  + o\big(\|\delta\Lm\|\big).
\]
\end{lemma}
This section is concluded by particularizing Lemma \ref{lemma:energy} for $\bP=\bJ$ and $\bP=\bK$.

\begin{proposition}\label{prop:geom:isot}
Consider a macroscopically isotropic medium in the sense of Definition \ref{def:macro:isot} and let $\beps$ denote the strain field solution for an applied macroscopic strain $\bbeps$. Then for any given phase $\phi$ the following properties hold:
\begin{itemize}
\item[--] If $\bbeps$ is purely spherical then $\langle\epse^2\rangle_\phi = o\big(\|\delta\Lm\|\big)$ and $\beps$ is such that
 \begin{equation}\label{expl:rela:J:geom:isot}
\kappa_\phi =  \kappa_0+\frac{\nJ \lJ^{-1}}{2d}\bigg[1- \frac{\langle\epsz^2\rangle_\phi}{\bepsz^2}\bigg] + o\big(\|\delta\Lm\|\big).
\end{equation}
\item[--] If $\bbeps$ is purely deviatoric then $\langle\epsz^2\rangle_\phi = o\big(\|\delta\Lm\|\big)$ and $\beps$ satisfies
\begin{equation}\label{expl:rela:K:geom:isot}
\mu_\phi  = \mu_0+\frac{\nK\lK^{-1}}{4}\bigg[1-\frac{\langle\epse^2\rangle_\phi}{{\bepse}^2}\bigg] + o\big(\|\delta\Lm\|\big).
\end{equation}
\end{itemize}
\end{proposition}

The results of Proposition \ref{prop:geom:isot} have been established for a piecewise-homogeneous and macroscopically isotropic medium. It is straightforward to extent them to the case of continuous material parameter distributions by reinterpreting Eqns. \eqref{derv:modulus} and \eqref{derv:Lmh} in terms of the Fréchet derivative of the energy functional, which corresponds to an infinitesimal material parameter perturbation at a given point $\bx$. Therefore, the identities \eqref{expl:rela:J:geom:isot} and \eqref{expl:rela:K:geom:isot} hold also true with $\kappa_\phi$, $\mu_\phi$ and $\langle\epsz^2\rangle_\phi$, $\langle\epse^2\rangle_\phi$ respectively superseded by the local quantities $\kappa(\bx)$, $\mu(\bx)$ and $\epsz(\bx)^2$, $\epse(\bx)^2$.

\begin{remark}
On noting that Lemma \ref{lemma:orth:para} implies the following first-order identity
\[
\frac{\eps\para(\bx)}{\|\bbeps\|}=\frac{1}{2}\left(1+\frac{\eps\para(\bx)^2}{\|\bbeps\|^2}\right),
\]
and using Eqns. \eqref{eq:epspara:P} and \eqref{def:eps:z:eq}, then Proposition \ref{prop:geom:isot} is consistent with propositions \ref{prop:sph:rela} and \ref{prop:dev:rela} given the fact that the macroscopically isotropic nature of the material configuration considered makes it redundant to use multiple strain field solutions. Indeed, if the assumption of macroscopic isotropy is satisfied then the elastic response of the medium is identical for rotated directions of applied strain. Whether this property holds or not for a given material sample can be assessed directly in practice. As a consequence, for such geometries only one strain field solution is actually sufficient to identify each one of the moduli $\kappa(\bx)$ and $\mu(\bx)$.
\end{remark}

\section{Case of bounded domains}\label{sec:bounded:dom}

This section investigates the case of a macro-homogeneous and bounded elastic medium. Consistently with the previous developments, it is assumed that the displacement field satisfies $\bm{u}(\bx)=\bbeps\cdot\bx$ on $\partial \mathcal{V}$ with given $\bbeps\in\sots \mathbb{R}^d$, so that the applied mechanical loading is compatible with a macroscopic strain satisfying the mean value property \eqref{e:mean}.

For such configurations, the integral formulation \eqref{LS} holds with the periodic Green's operator $\bGz$ being superseded by the Green's operator $\bG$ for the bounded geometry $\mathcal{V}$ and the elasticity tensor $\Lmz$. Following \cite{Willis1977,Willis1981}, the assumption that $\mathcal{V}$ is large compared to the length-scale of the inhomogeneities, allows to make the following translation-invariant approximation for $\bx$ distant from $\partial \mathcal{V}$:
\begin{equation}\label{trans:inv:approx}
\big[\bG \btau \big](\bx)=\int_{\mathcal{V}} \bGzi(\bx-\by):\big[\btau(\by)-\langle\btau\rangle\big]\td\by
\end{equation}
with $\bGzi:\mathcal{V}\to\sots \!\big(\!\sots\! \mathbb{R}^d\big)$ being the infinite-body Green's function associated with $\Lmz$ and vanishing displacements at infinity. Note that the kernel $\bGzi$ being singular then evaluating the above integral requires special treatment. This approximation relies on the decaying behavior of the Green's function and it amounts to neglect the contribution of boundary terms to the integral operator $\bG$ by assuming that they are significant only in a region close to $\partial\mathcal{V}$. With equation \eqref{trans:inv:approx} at hand, then deriving the sought elasticity-based local identities can be achieved by characterizing the local and non-local contributions to this integral operator as discussed in Remark \ref{rmk:l:nl}. 

Alternatively, the approach adopted hereafter is built on the Fourier-based analysis of Section~\ref{sec:strain:approach} so as to take full advantage of the orthogonal decomposition of Lemma \ref{def:decomp:l:nl}. The aim is to show what partial differential equation (PDE) can be derived and solved for bounded domain configurations in order to relate locally the elastic moduli to the strain field data. This is done using an approach that is analogous to the one employed to solve the auxiliary problem \eqref{eq:ref:eps:tau}. The procedure investigated hereafter relies on a reinterpretation of the Fourier algebra in terms of partial differential operators. 

\subsection{Derivation of a PDE for the sought material parameters}\label{PDE:parameters}

Starting from the equilibrium equation in \eqref{eq:elas}, then for all $\bx\in\mathcal{V}$ one has at the first order
\begin{equation}\label{div:id}
\nab\cdot\big(\delta\Lm(\bx):\bbeps\big)=-\nab\cdot\big(\Lmz:\delta\beps(\bm{x})\big)
\end{equation}
where $\delta\Lm(\bx)=\Lm(\bx)-\Lmz$ and $\delta\beps(\bx)=\beps(\bx)-\bbeps$. Moreover, let introduce a displacement field $\bm{u}$ such that $\delta\beps(\bx)=\nab\!\sym \bm{u}(\bx)$. Assuming that the macroscopic strain $\bbeps$ is either purely spherical or purely deviatoric, i.e. $\bP\dc\bbeps=\bbeps$ with either $\bP=\bJ$ or $\bP=\bK$, then $\delta\Lm(\bx):\bbeps=\delta p(\bx)\,\bbeps$ with $\delta p(\bx)=d\,\delta\kappa(\bx)$ or $\delta p(\bx)=2\,\delta \mu(\bx)$ respectively. With the terms in \eqref{div:id} reducing to
\[\left\{\begin{aligned}
& \nab\cdot\big(\delta\Lm:\bbeps\big)=\nab \delta p\cdot \bbeps, \\
& \nab\cdot(\Lmz:\delta\beps)=\omega_0\nab\nab\cdot\bm{u}+\mu_0\,\Delta\bm{u} \qquad \text{where} \qquad \omega_0=\frac{d\,\kappa_0+(d-2)\mu_0}{d}.
\end{aligned}\right.\]
one obtains
\begin{equation}\label{eq0}
-\nab \delta p\cdot \bbeps=\omega_0\nab\nab\cdot\bm{u}+\mu_0\,\Delta\bm{u}.
\end{equation}
Applying the divergence operator to \eqref{eq0} yields 
\begin{equation}\label{eq1}
-\nab^{\otimes2}\delta p:\bbeps=(\omega_0+\mu_0)\Delta\nab\cdot\bm{u}.
\end{equation}
since $\nab\cdot(\nab \delta p\cdot \bbeps)=\nab\nab \delta p : \bbeps$, a term that can be rewritten in short using the tensor given in components as $(\nab^{\otimes2}\delta p)_{ij}=\delta p_{,ij}$ in terms of partial derivatives with respect to coordinates $i$ and $j$. Now, by applying the Laplace operator to \eqref{eq0} and using the fact that it commutes with the gradient operator we obtain
\[
-\Delta(\nab \delta p\cdot \bbeps) =\omega_0\,\nab\Delta\nab\cdot\bm{u}+\mu_0\,\Delta\Delta\bm{u}.
\]
Substituting the term $\Delta\nab\cdot\bm{u}$ from \eqref{eq1} in the above equation entails
\[
\Delta\Delta\bm{u}=-\frac{1}{\mu_0}\Delta(\nab \delta p\cdot \bbeps)+\tau_0\nab(\nab^{\otimes2}\delta p:\bbeps) \qquad \text{with} \qquad \tau_0=\frac{\omega_0}{\mu_0(\omega_0+\mu_0)}.
\]
Finally, taking the gradient of the previous identity yields
\begin{equation}\label{eq2}
\Delta\Delta \nab\bm{u}=-\frac{1}{\mu_0}\Delta\nab^{\otimes2} \delta p\cdot \bbeps + \tau_0\nab^{\otimes4}\delta p:\bbeps,
\end{equation}
where the fourth-order tensor $\nab^{\otimes4}\delta p$ is defined in components by $(\nab^{\otimes4}\delta p)_{ijk\ell}=\delta p_{,ijk\ell}$.\\

We now consider a set of strain field solutions $\beps^{(i)}$ satisfying $\langle\beps^{(i)}\rangle=\bbeps^{(i)}$ for $i=1,\dots,\nP$ and we write $\nab\!\sym\bm{u}^{(i)}=\delta\beps^{(i)}=\beps^{(i)}-\bbeps^{(i)}$. By symmetry we have that $\nab\bm{u}^{(i)}:\bbeps^{(i)}=\delta\beps^{(i)}:\bbeps^{(i)}$ so that applying the inner product with $\bbeps^{(i)}/\|  \bbeps^{(i)} \|^2$ in \eqref{eq2} and summing for $i=1,\dots,\nP$ leads to
\[
\Delta\Delta\sum_{i=1}^{\nP}\frac{\delta\beps^{(i)}:\bbeps^{(i)}}{\|  \bbeps^{(i)} \|^2} = \sum_{i=1}^{\nP}\Big\{ -\frac{1}{\mu_0}\Delta ( \nab^{\otimes2} \delta p\cdot \bbeps^{(i)} ) : \bbeps^{(i)} + \tau_0\nab^{\otimes4} \delta p::\bbeps^{(i)}\otimes\bbeps^{(i)} \Big\} \frac{1}{\|  \bbeps^{(i)} \|^2}.
\]
As in Section \ref{sec:strain:approach}, if the tensors $\bbeps^{(i)}\in\sots\mathbb{R}^d$ satisfy Eqn. \eqref{basis:P}, i.e. they constitute an orthogonal basis for the fourth-order tensor projector $\bm{\mathcal{P}}$ with $\bm{\mathcal{P}}=\bm{\mathcal{J}}$ or $\bm{\mathcal{K}}$, then the previous equation can finally be recast as follows
\begin{equation}\label{eq:generic}
\Delta\Delta\sum_{i=1}^{\nP}\frac{\delta\beps^{(i)}:\bbeps^{(i)}}{\|  \bbeps^{(i)} \|^2} =  -\frac{1}{\mu_0}\Delta  \nab^{\otimes2} \delta p : \tr_{23}[\bm{\mathcal{P}}] + \tau_0\nab^{\otimes4}\delta p::\bm{\mathcal{P}},
\end{equation}
where $\tr_{23}[\bm{\mathcal{P}}]$ is the second-order tensor with components $\big(\!\tr_{23}[\bm{\mathcal{P}}]\big)_{ij}=\big(\bm{\mathcal{P}}\big)_{ikkj}$ where index $k$ is summed. The identity \eqref{eq:generic} constitutes the sought PDE for the unknown material parameter $\delta p$. It is now particularized for spherical and deviatoric macroscopic strains in the next two subsections.

\subsection{Spherical macroscopic strain}

Considering the case $\bm{\mathcal{P}}=\bm{\mathcal{J}}$ with $\nJ=1$ and $\delta p(\bx)=d\,\delta\kappa(\bx)$, then in Eqn. \eqref{eq:generic} we have
\[
 \nab^{\otimes2} \delta p : \tr_{23}[\bm{\mathcal{J}}]=\frac{1}{d}\Delta \delta p \qquad \text{and} \qquad \nab^{\otimes4}\delta p::\bm{\mathcal{J}}=\frac{1}{d}\Delta\Delta \delta p
\]
so that 
\[
\Delta\Delta\frac{\delta\beps:\bbeps}{\|  \bbeps \|^2}=\left(\frac{\tau_0\mu_0-1}{d\mu_0}\right)\Delta\Delta \delta p.
\]
After integration and substituting the definitions of $\tau_0$ and $\omega_0$, this equation yields for all $\bx\in\mathcal{V}$:
\begin{equation}\label{form:kappa}
\delta\kappa(\bx)=-\left(\frac{d\kappa_0+2(d-1)\mu_0}{d}\right)\frac{\delta\beps(\bx):\bbeps}{\|  \bbeps \|^2} + k(\bx),
\end{equation}
where the field $k$ is an arbitrary function that is biharmonic in $\mathcal{V}$, i.e. $\Delta\Delta k=0$. This shows that the analytical solution obtained in Proposition \ref{prop:sph:rela} for periodic configurations is only a particular solution in a bounded domain. In such a case, to determine the elastic modulus uniquely one has to take into account additional boundary conditions, which can be derived by assessing the behaviors of $\delta\kappa$ and $\delta\beps$ and of their derivatives on $\partial \mathcal{V}$. In particular, boundary conditions that are suitable to a unique reconstruction could be obtained when the investigated sample is embedded in a homogenous and isotropic elastic matrix.

\subsection{Deviatoric macroscopic strains}

We now investigate the case $\bm{\mathcal{P}}=\bm{\mathcal{K}}$ with $\nK=\frac{d(d+1)}{2}-1$ and $\delta p(\bx)=2\,\delta\mu(\bx)$. In Eqn. \eqref{eq:generic} we have
\[
\nab^{\otimes2} \delta p : \tr_{23}[\bm{\mathcal{K}}]=\frac{(d-1)(d+2)}{2d}\Delta \delta p \qquad \text{and}\qquad \nab^{\otimes4}\delta p::\bm{\mathcal{K}}=\frac{(d-1)}{d}\Delta\Delta \delta p
\]
so that 
\[
\Delta\Delta\sum_{i=1}^{\nK}\frac{\delta\beps^{(i)}:\bbeps^{(i)}}{\|  \bbeps^{(i)} \|^2} =  \frac{(d-1)}{2d\mu_0}\big(2\tau_0\mu_0-d-2\big)\Delta\Delta \delta p.
\]
Integrating the bi-Laplace operator and based on the definitions of $\tau_0$ and $\omega_0$, then the left and right hand sides in the above equation turn out to be equal up to an additional and arbitrary function $m(\bx)$ that is biharmonic in $\mathcal{V}$, i.e. for all $\bx\in\mathcal{V}$:
\begin{equation}\label{form:mu}
\delta\mu(\bx)=-\frac{\mu_0\big(d\kappa_0+2(d-1)\mu_0\big)}{d(d-1)(\kappa_0+2\mu_0)}\sum_{i=1}^{\nK}\frac{\delta\beps^{(i)}(\bx):\bbeps^{(i)}}{\|  \bbeps^{(i)} \|^2} +  m(\bx).
\end{equation}
This shows again that the reconstruction formula of Proposition \ref{prop:dev:rela} yields a particular solution in a bounded domain while a unique reconstruction can be achieved when appropriate boundary conditions are met.

\section{Analytical examples}\label{sec:analytical:ex}

In this section, two analytical examples are investigated to illustrate the results of propositions \ref{prop:sph:rela}--\ref{prop:geom:isot}.

\subsection{Spherical inclusion}

The simplest analytical example is provided by the case of an unbounded elastic matrix with moduli $\kappa_0$, $\mu_0$, containing a spherical homogeneous inclusion $S_R$ with radius $R$, moduli $\kappa_1=\kappa_0+\delta\kappa$ and $\mu_1=\mu_0+\delta\mu$, and subjected to an applied strain $\bbeps$ at infinity. It is well known, see \cite{Eshelby,Torquato}, that the strain field solution $\beps(\bx)$ is uniform within the inclusion, i.e. $\beps(\bx)=\beps$ for $|\bx|<R$ with
\begin{equation}\label{sol:sphere}
\beps=\big[ (1-\kappa_s)\bJ+(1-\mu_s)\bK \big]:\bbeps,
\end{equation}
where 
\begin{equation}\label{def:par:sphere}
\kappa_s=\frac{\delta\kappa}{\kappa_1+\frac{2(d-1)}{d}\mu_0}, \qquad \mu_s=\frac{\delta\mu}{\mu_1+\theta_0}, \qquad \theta_0=\mu_0\frac{d^2\kappa_0+2(d+1)(d-2)\mu_0}{2d(\kappa_0+2\mu_0)}.
\end{equation}
Upon prescribing a purely spherical macroscopic strain $\bbeps$ that satisfies $\bJ\dc\bbeps=\bbeps$ and $\bK\dc\bbeps=\bm{0}$, then one obtains that $\dev[\beps]=\bm{0}$ within $S_R$. Moreover expanding Eqn. \eqref{sol:sphere} at the first order entails
\begin{equation}\label{sphere:sph}
\kappa_1 =  \kappa_0+\frac{d\kappa_0+2(d-1)\mu_0}{d}\bigg[1- \frac{\tr[\beps]\tr[\bbeps]}{d\|\bbeps\|^2}\bigg] + o(|\delta\kappa|).
\end{equation}
Likewise, when the applied strain is purely deviatoric, i.e. $\bK\dc\bbeps=\bbeps$ and $\bJ\dc\bbeps=\bm{0}$, then the strain field satisfies $\tr[\beps]=0$ in $S_R$ and one has
\begin{equation}\label{sphere:dev}
\mu_1=\mu_0+\frac{(d+2) \mu_0 \big( d\kappa_0+2(d-1)\mu_0 \big)}{2d(\kappa_0+2\mu_0)}\bigg[1- \frac{\dev[\beps]\dc\dev[\bbeps]}{\|\bbeps\|^2}\bigg] + o(|\delta\mu|)
\end{equation}
Therefore, owing to the definitions of $\nJ$, $\nK$ in \eqref{id:qc} and of $\lJ^{-1}$, $\lK^{-1}$ in Section \ref{sec:strain:approach}, one can conclude that the identities \eqref{sphere:sph} and \eqref{sphere:dev} coincides with the results of propositions \ref{prop:sph:rela} and \ref{prop:dev:rela} given that only one strain field is sufficient in each case to derive the sought strain-modulus identities. Consistently with the developments of Section \ref{sec:energy:approach}, this result relies on the Proposition \ref{prop:geom:isot} which applies to the isotropic geometry of the spherical inclusion problem.

Finally, note that these results can be extended to the case of a suspension of spherical elastic inclusions provided that mutual interactions can be neglected, i.e. as long as the strain field can be considered to be homogeneous within each spheres and given by Eqn. \eqref{sol:sphere}.

\subsection{Macroscopic isotropy and Hashin-Shtrikman bounds}

A second analytical example illustrating the results of Proposition \ref{prop:geom:isot} is obtained by considering a two-phase macroscopically isotropic microstructure, that is not intended to be described here, but for which one of the Hashin-Shtrikman bounds is attained \cite{Hashin,Milton}. Let consider weak-contrast phases $\phi=1$, $2$ such that for $\eta_\phi=\kappa_\phi$ or $\eta_\phi=\mu_\phi$ one has
\begin{equation}\label{ph:weak}
\eta_2 = \eta_0 + f_1 \delta\eta \qquad \text{and}\qquad \eta_1 = \eta_0 - f_2 \delta\eta
\end{equation}
with $\delta\eta=o(\eta_0)$. Since $f_1+f_2=1$, this definition ensures that $\langle\eta\rangle=f_1\eta_1+f_2\eta_2=\eta_0$ and $\eta_2-\eta_1=\delta\eta$ so that the corresponding elasticity tensor of the form \eqref{Ela:comp} satisfies \eqref{def:elas:weak}. According to the expressions of the Hashin-Shtrikman bounds then one further assumes that the isotropic effective elasticity tensor $\Lmh=d \tilde{\kappa} \bJ + 2\tilde{\mu}\bK$ is given by
\begin{equation}\label{HS}
\tilde{\kappa} = \langle\kappa\rangle - \frac{f_1f_2(\kappa_1-\kappa_2)^2}{{\langle\kappa\rangle}_{\!*} + 2\frac{(d-1)}{d}\mu_1} \qquad \text{and} \qquad
\tilde{\mu} = \langle\mu\rangle - \frac{f_1f_2(\mu_1-\mu_2)^2}{{\langle\mu\rangle}_{\!*} + \theta_1}
\end{equation}
where ${\langle\eta\rangle}_{\!*}=f_2\eta_1+f_1\eta_2$ and $\theta_1$ defined as in Eqn. \eqref{def:par:sphere} with the reference moduli $\kappa_0$, $\mu_0$ replaced by these of phase $1$. Based on Equation \eqref{derv:modulus} and choosing either $t=\kappa_1$ or $t=\mu_1$ one obtains 
\begin{equation}\label{derv:en:isot:rela}\left\{\begin{aligned}
& f_1 {\langle\epsz^2\rangle}_1=  \bepsz^2 \frac{\p\tilde{\kappa}}{\p \kappa_1} + \frac{2}{d(d-1)}\bepse^2\frac{\p\tilde{\mu}}{\p\kappa_1}\\
& f_1  {\langle\epse^2\rangle}_1 = \frac{d(d-1)}{2} \bepsz^2 \frac{\p\tilde{\kappa}}{\p \mu_1} + \bepse^2\frac{\p\tilde{\mu}}{\p\mu_1} .
\end{aligned}\right.\end{equation}
Using Eqns. \eqref{ph:weak} and \eqref{HS}, the material derivatives entering \eqref{derv:en:isot:rela} is expanded at the first order as
\[
 \frac{\p\tilde{\kappa}}{\p \kappa_1} =  f_1 + \frac{2f_1f_2 \, \delta\kappa}{\kappa_0 + 2\frac{(d-1)}{d}\mu_0} + o(\|\delta\Lm\|), \qquad \frac{\p\tilde{\kappa}}{\p \mu_1}=o(\|\delta\Lm\|).
\]
and 
\[
\frac{\p\tilde{\mu}}{\p \kappa_1}=o(\|\delta\Lm\|), \qquad \frac{\p\tilde{\mu}}{\p \mu_1}  = f_1 + \frac{2f_1f_2 \, \delta\mu}{\mu_0 + \theta_0} + o(\|\delta\Lm\|).
\]
where $\theta_0$ is defined as in Eqn. \eqref{def:par:sphere}. 

Therefore, when the macroscopic strain is purely spherical, i.e. $\bepse=0$, then according to \eqref{derv:en:isot:rela} one obtains ${\langle\epse^2\rangle}_1=o(\|\delta\Lm\|)$ and
\[
 \frac{\langle\epsz^2\rangle_1}{\bepsz^2} = 1 + \frac{2f_2 \, \delta\kappa}{\kappa_0 + 2\frac{(d-1)}{d}\mu_0} + o(\|\delta\Lm\|),
\]
which, using \eqref{ph:weak}, finally yields
\begin{equation}\label{ph:sph:rela}
\kappa_1=\kappa_0+\frac{d\kappa_0+2(d-1)\mu_0}{2d}\bigg[1- \frac{\langle\epsz^2\rangle_1}{\bepsz^2}\bigg] + o(\|\delta\Lm\|).
\end{equation}
Conversely, if the macroscopic strain satisfies $\bepsz=0$, then Eqns. \eqref{derv:en:isot:rela} yields ${\langle\epsz^2\rangle}_1=o(\|\delta\Lm\|)$ and
\begin{equation}\label{ph:dev:rela}
\mu_1=\mu_0+\frac{(d+2) \mu_0 \big( d\kappa_0+2(d-1)\mu_0 \big)}{4d(\kappa_0+2\mu_0)}\bigg[1- \frac{\langle\epse^2\rangle_1}{\bepse^2}\bigg] + o(\|\delta\Lm\|) .
\end{equation}
According to the definitions of $\nJ$, $\nK$ in \eqref{id:qc} as well as of $\lJ^{-1}$, $\lK^{-1}$ in propositions \ref{prop:sph:rela} and \ref{prop:dev:rela}, then one can conclude that the identities \eqref{ph:sph:rela} and \eqref{ph:dev:rela} coincides with the results of Proposition \ref{prop:geom:isot}.

\section{Numerical examples}\label{sec:numerical:ex}

In this section we present a set of numerical results for 2D material configurations relative to the bounded domain $\mathcal{V}=[0,1]\times[0,1]$. The domain is meshed based on a regular grid of $500\times500$ nodes which defines a structured mesh of triangular elements. In particular, the mesh considered is independent of the elasticity parameter distributions that are discussed subsequently. Reference material parameter maps $\kappa_\text{ref}(\bx)$, $\mu_\text{ref}(\bx)$ are then generated as uniform random distributions with mean values $\kappa_0=\mu_0=1$ and amplitude parametrized by the contrast value $c$. Configuration 1 in Fig. \ref{fig:mat:parA} corresponds to a smooth and geometrically anisotropic material distribution. Configuration 2 shown in Figure \ref{fig:mat:par} is defined as an arrangement of piecewise homogenous phases obtained by Voronoi tesselation. 

\begin{figure}[hbt]	
\centering
\subfloat[$\kappa_\text{ref}(\bx)-\kappa_0$]{\includegraphics[height=0.21\textheight]{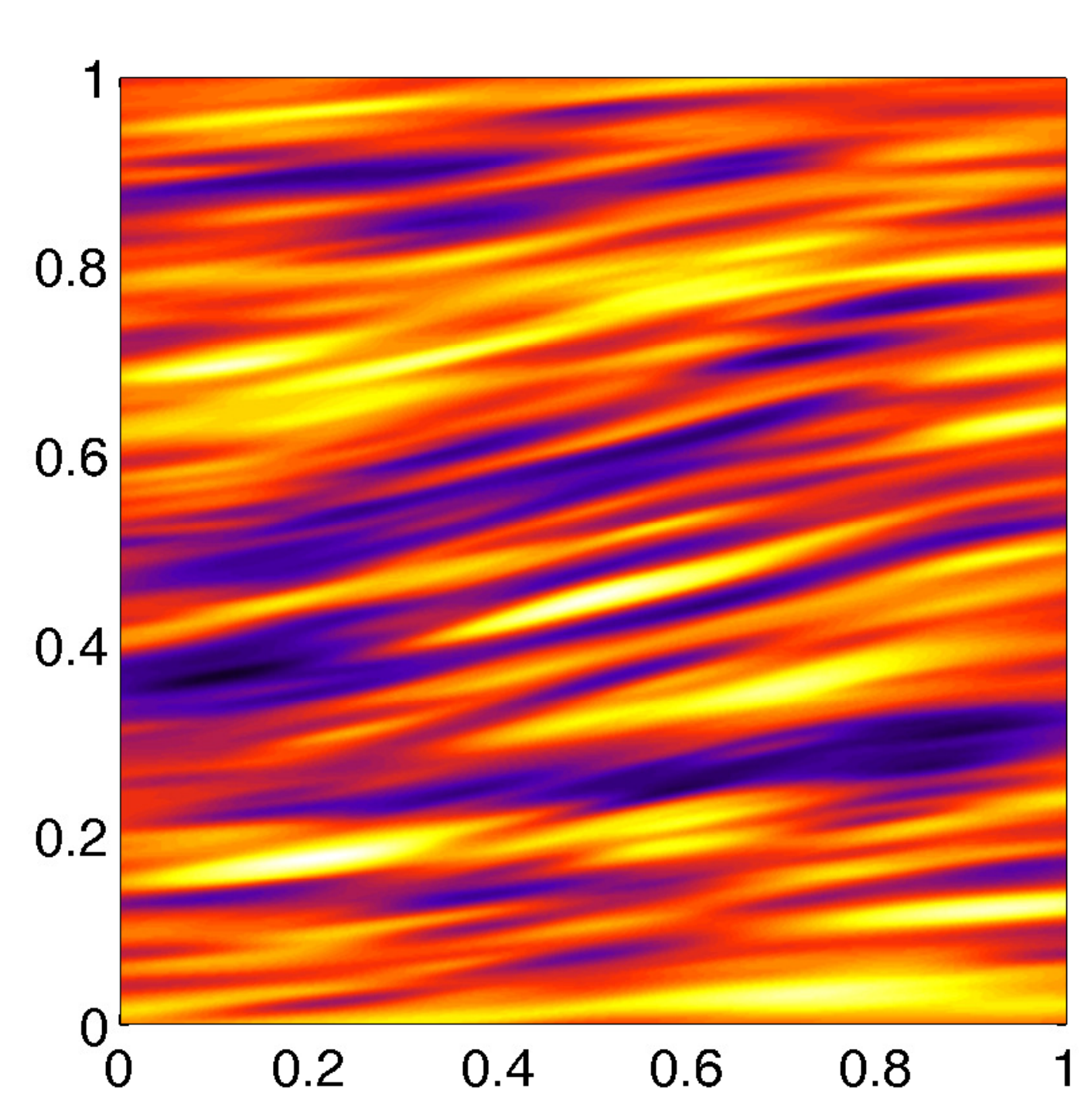}\label{kappadataA}}\hspace{5mm}
\subfloat[$\mu_\text{ref}(\bx)-\mu_0$]{\includegraphics[height=0.21\textheight]{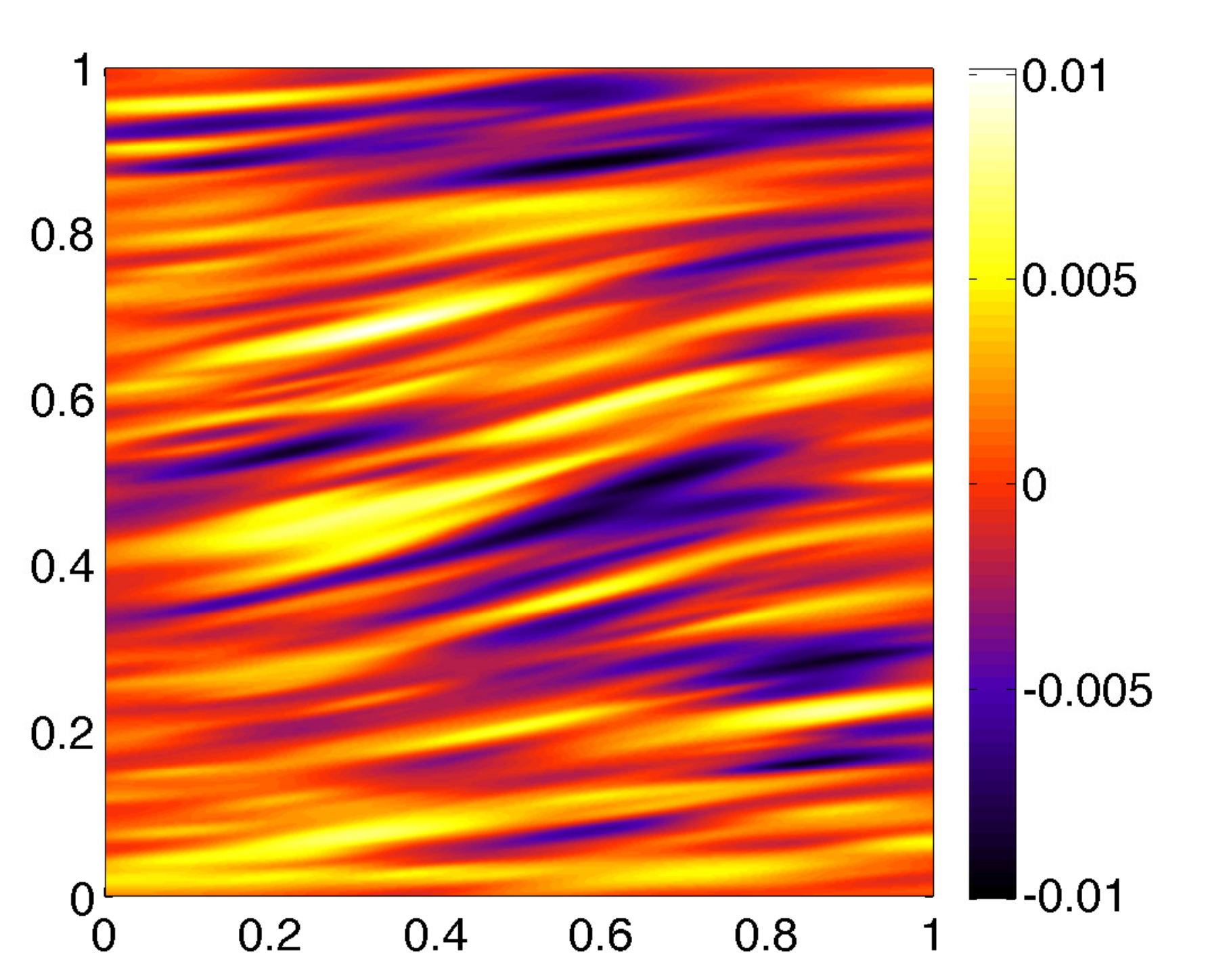}\label{mudataA}}
\caption{\emph{Configuration 1:} Reference elasticity parameter distributions with $c=10^{-2}$.}
\label{fig:mat:parA}
\end{figure}

\begin{figure}[thb]	
\centering
\subfloat[$\kappa_\text{ref}(\bx)-\kappa_0$]{\includegraphics[height=0.21\textheight]{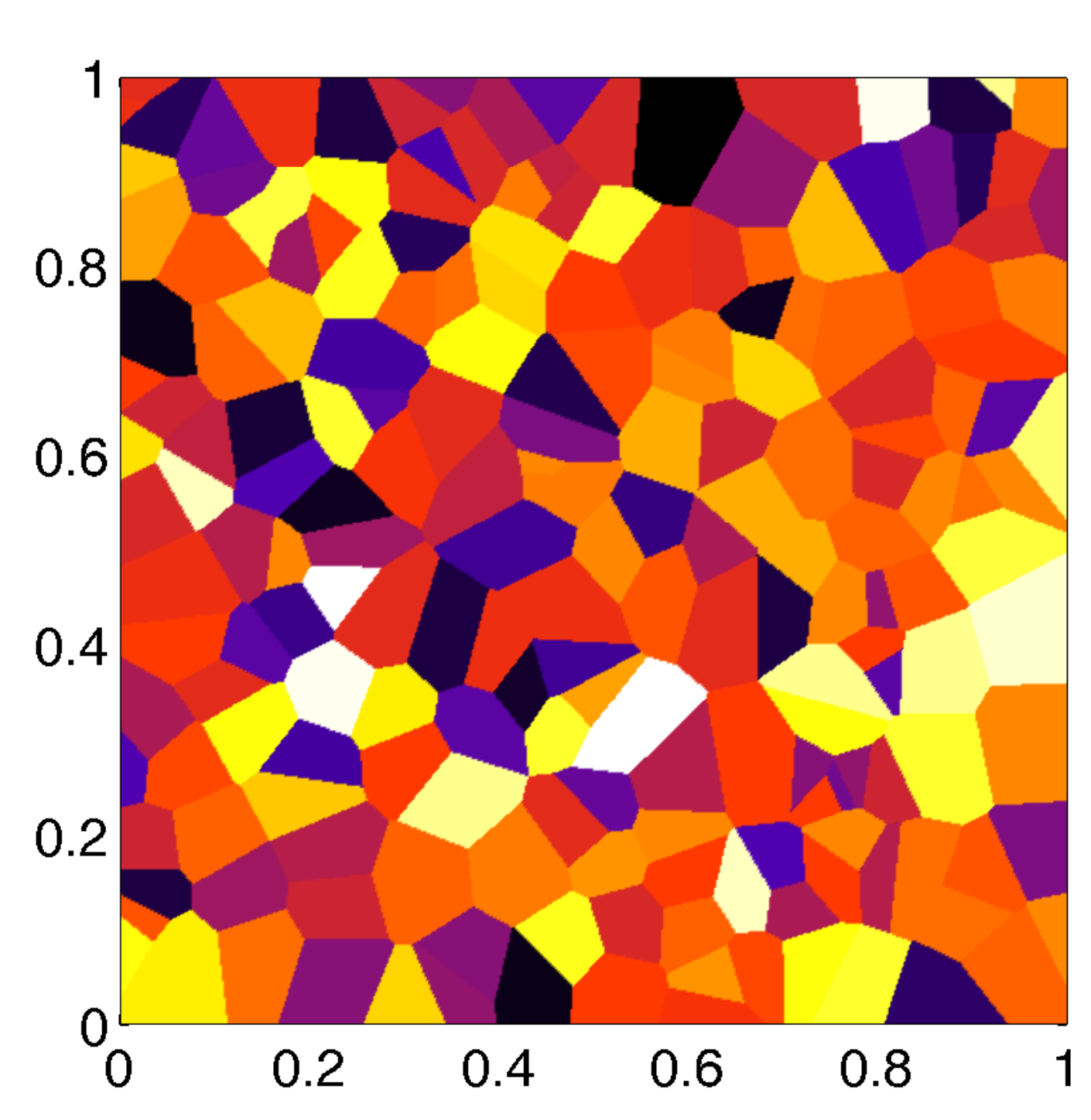}\label{kappadata}}\hspace{5mm}
\subfloat[$\mu_\text{ref}(\bx)-\mu_0$]{\includegraphics[height=0.21\textheight]{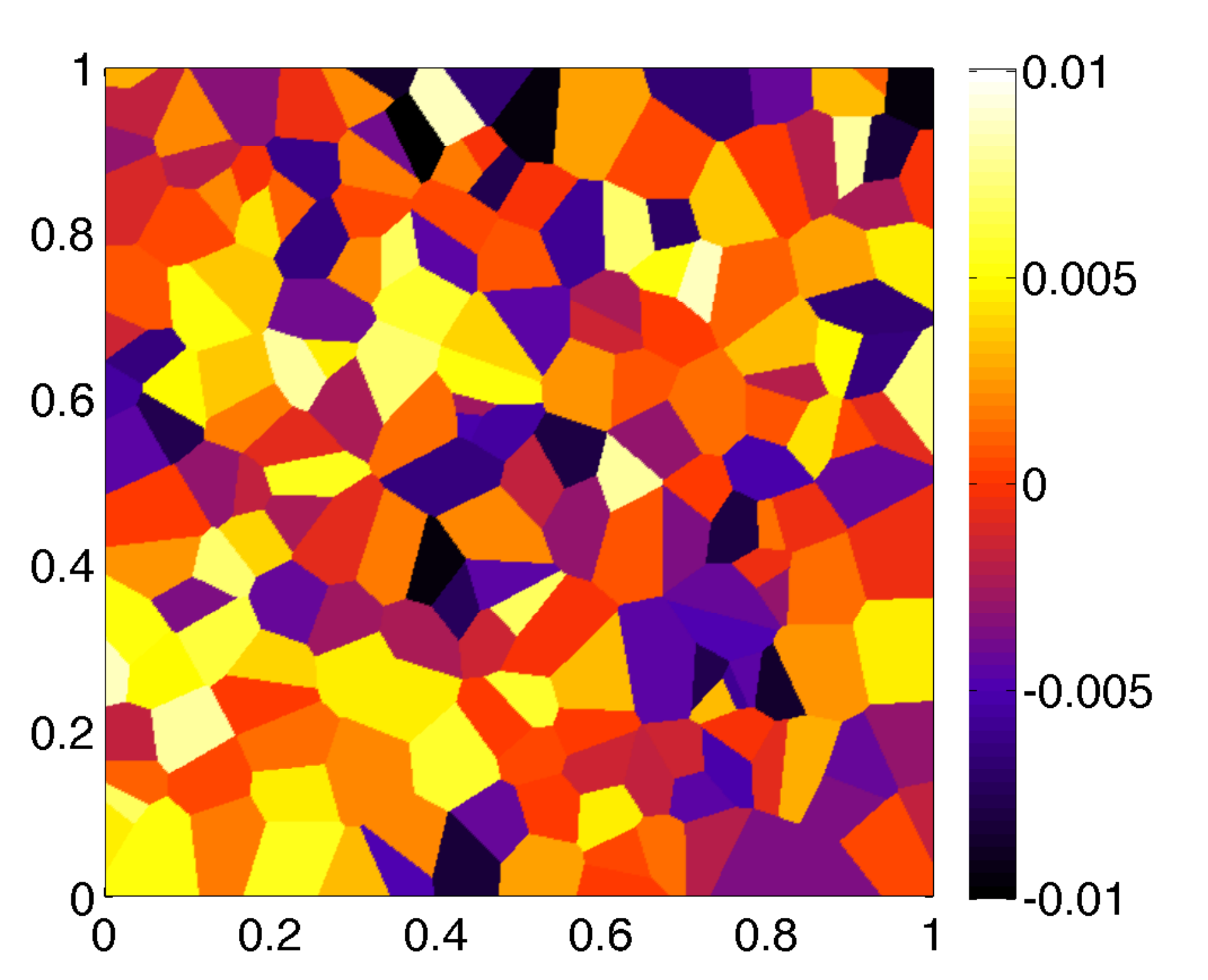}\label{mudata}}
\caption{\emph{Configuration 2:} Reference elasticity parameter distributions with $c=10^{-2}$.}
\label{fig:mat:par}
\end{figure}

According to the previous developments then in 2D the knowledge of three strain field solutions is required to be able to reconstruct the reference elasticity parameter fields $\kappa_\text{ref}(\bx)$, $\mu_\text{ref}(\bx)$. These solutions denoted as $\beps^{(i)}(\bx)$ for $i\in\{1,\,2,\,3\}$ are defined by prescribing linear displacements on the domain boundary $\partial\mathcal{V}$ as $\bm{u}^{(i)}(\bx)=\bbeps^{(i)}\cdot\bx$ with
\[\bbeps^{(1)}=\bm{e}_1\otimes\bm{e}_1+\bm{e}_2\otimes\bm{e}_2,\quad
\bbeps^{(2)}=\bm{e}_1\otimes\bm{e}_2+\bm{e}_2\otimes\bm{e}_1, \quad
 \bbeps^{(3)}=\bm{e}_1\otimes\bm{e}_1-\bm{e}_2\otimes\bm{e}_2.
\]
These macroscopic strains are chosen so as to satisfy $\bK\dc\bbeps^{(1)}=\bm{0}$ and $\bJ\dc\bbeps^{(2)}=\bJ\dc\bbeps^{(3)}=~\bm{0}$.

Displacement solutions $\bm{u}^{(i)}$ and associated strain fields $\beps^{(i)}$ are computed for the chosen reference configurations using standard linear finite elements. Then the fields $\kappa_\text{ref}(\bx)$, $\mu_\text{ref}(\bx)$ are considered to be unknown and the strain field values $\beps^{(i)}(\bx)$ computed at each node constitute in turn the full-field data set for the inverse problem. The objective elastic moduli are reconstructed by converting the strain maps that are associated with these synthetic measurements based on the formulae derived in propositions \ref{prop:sph:rela}--\ref{prop:geom:isot}. Therefore, the elasticity maps are computed using only the information available from the finite element procedure that is performed on a structured mesh. We denote as $\kappa^{(j)}(\bx)$, $\mu^{(j)}(\bx)$ the reconstructed elasticity maps, with the superscript $(j)$ indicating the strain field data that is used for the computation. Color scales are adjusted to enable comparison between figures.

\subsection{Reconstruction examples}

\begin{figure}[bht]	
\centering
\subfloat[$\kappa^{(1)}(\bx)-\kappa_0$]{\includegraphics[height=0.21\textheight]{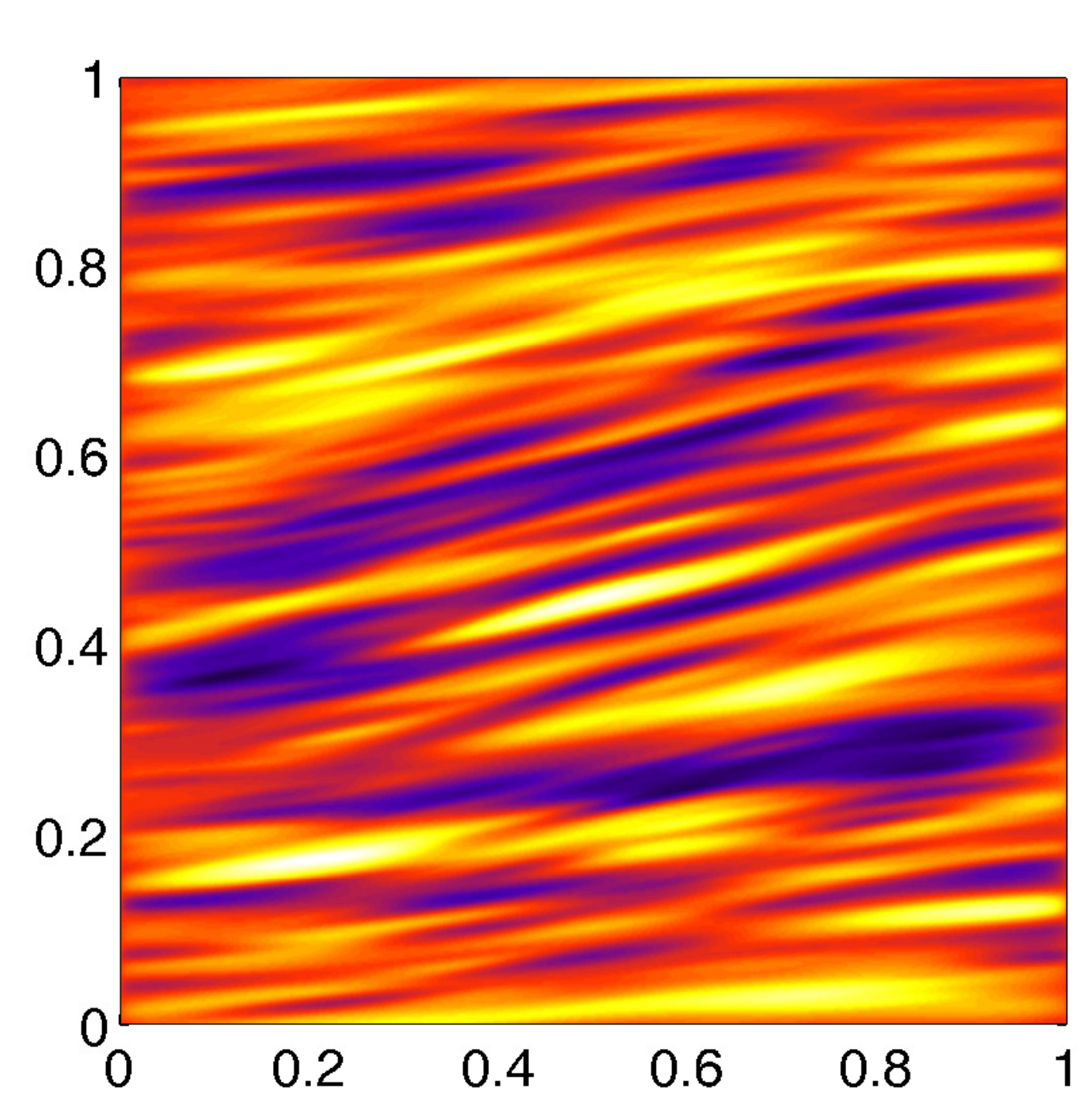}\label{kappaIdA}}\hspace{1mm}
\subfloat[$\mu^{(2)}(\bx)-\mu_0$]{\includegraphics[height=0.21\textheight]{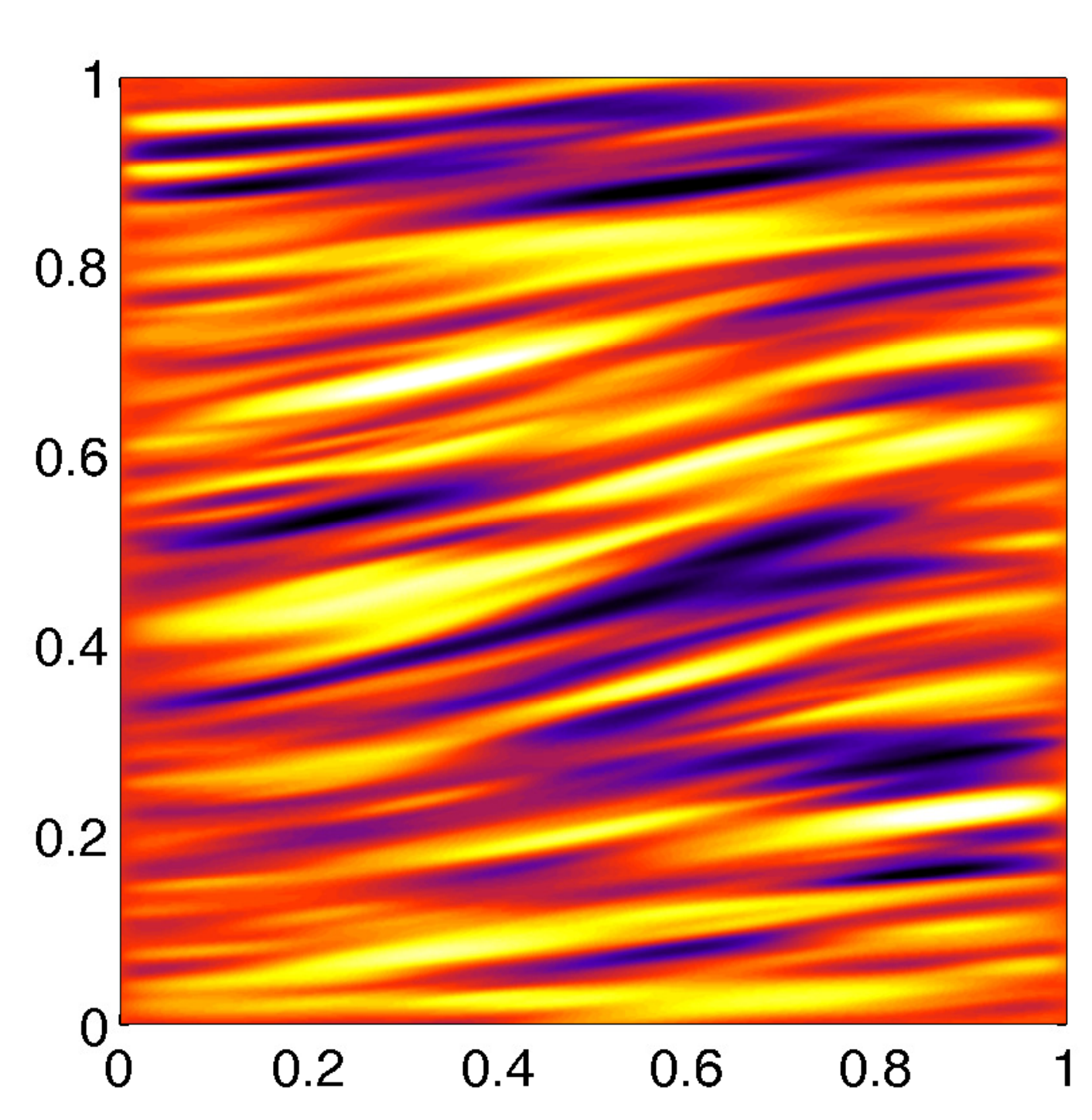}\label{muId2A}}\hspace{1mm}
\subfloat[$\mu^{(2,3)}(\bx)-\mu_0$]{\includegraphics[height=0.21\textheight]{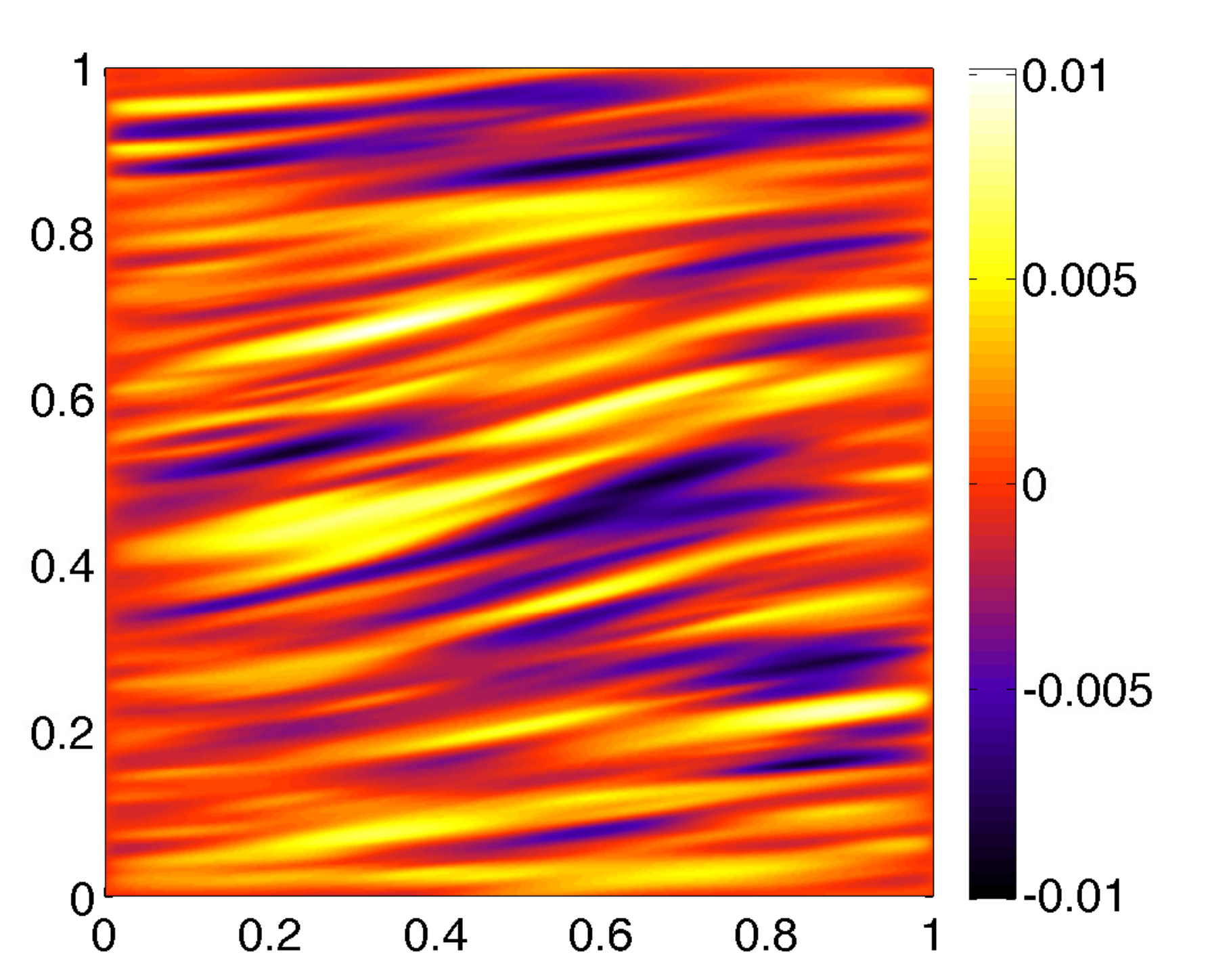}\label{muIdA}}
\caption{\emph{Configuration 1:} Reconstructed elasticity maps in the case where $c=10^{-2}$.}
\end{figure}

\begin{figure}[thb]	
\centering
\subfloat[$\displaystyle\frac{1}{c}|\kappa_\text{ref}(\bx)-\kappa^{(1)}(\bx)|$]{\includegraphics[height=0.21\textheight]{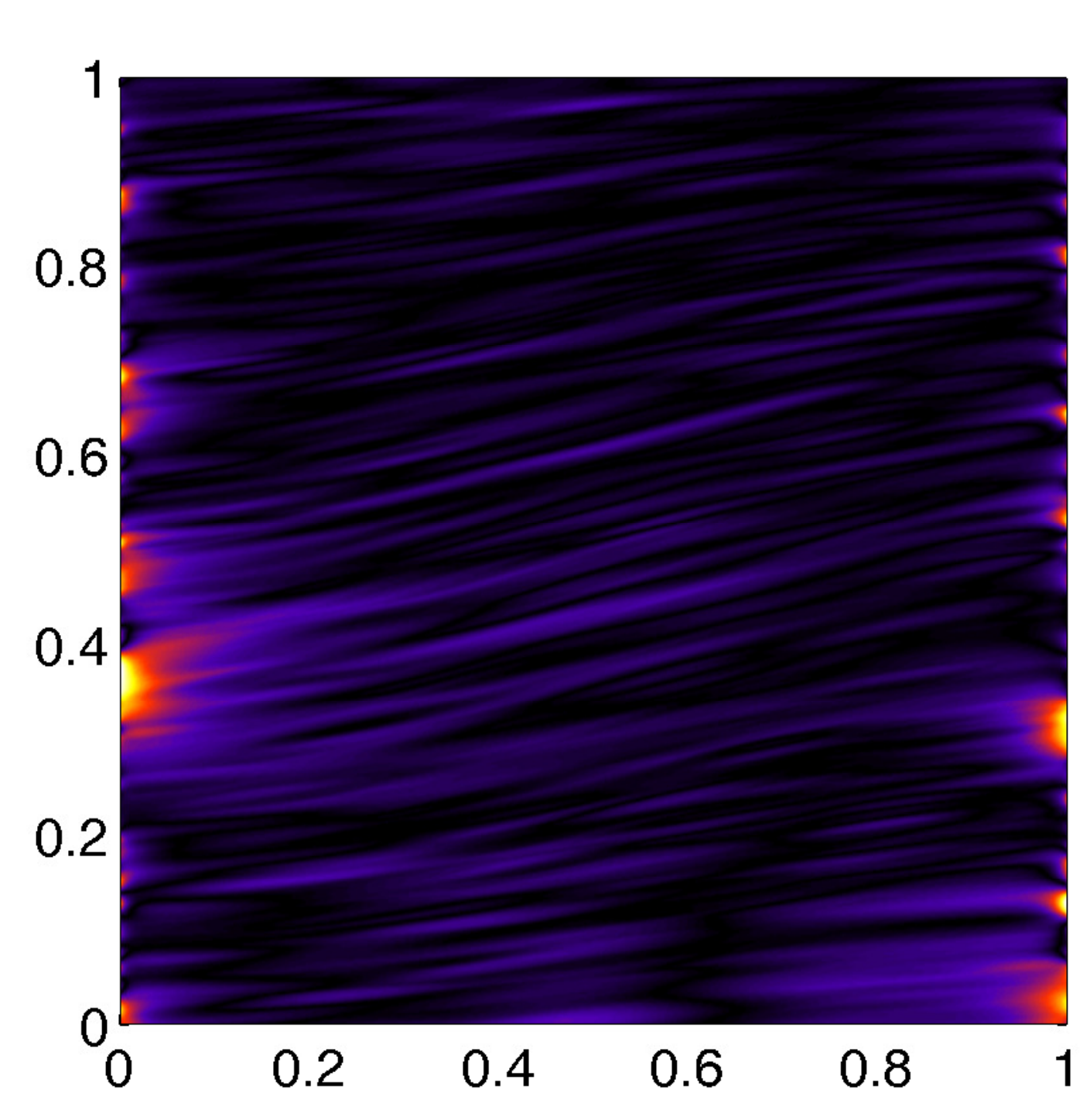}\label{ErkA}}\hspace{1mm}
\subfloat[$\displaystyle\frac{1}{c}|\mu_\text{ref}(\bx)-\mu^{(2)}(\bx)|$]{\includegraphics[height=0.21\textheight]{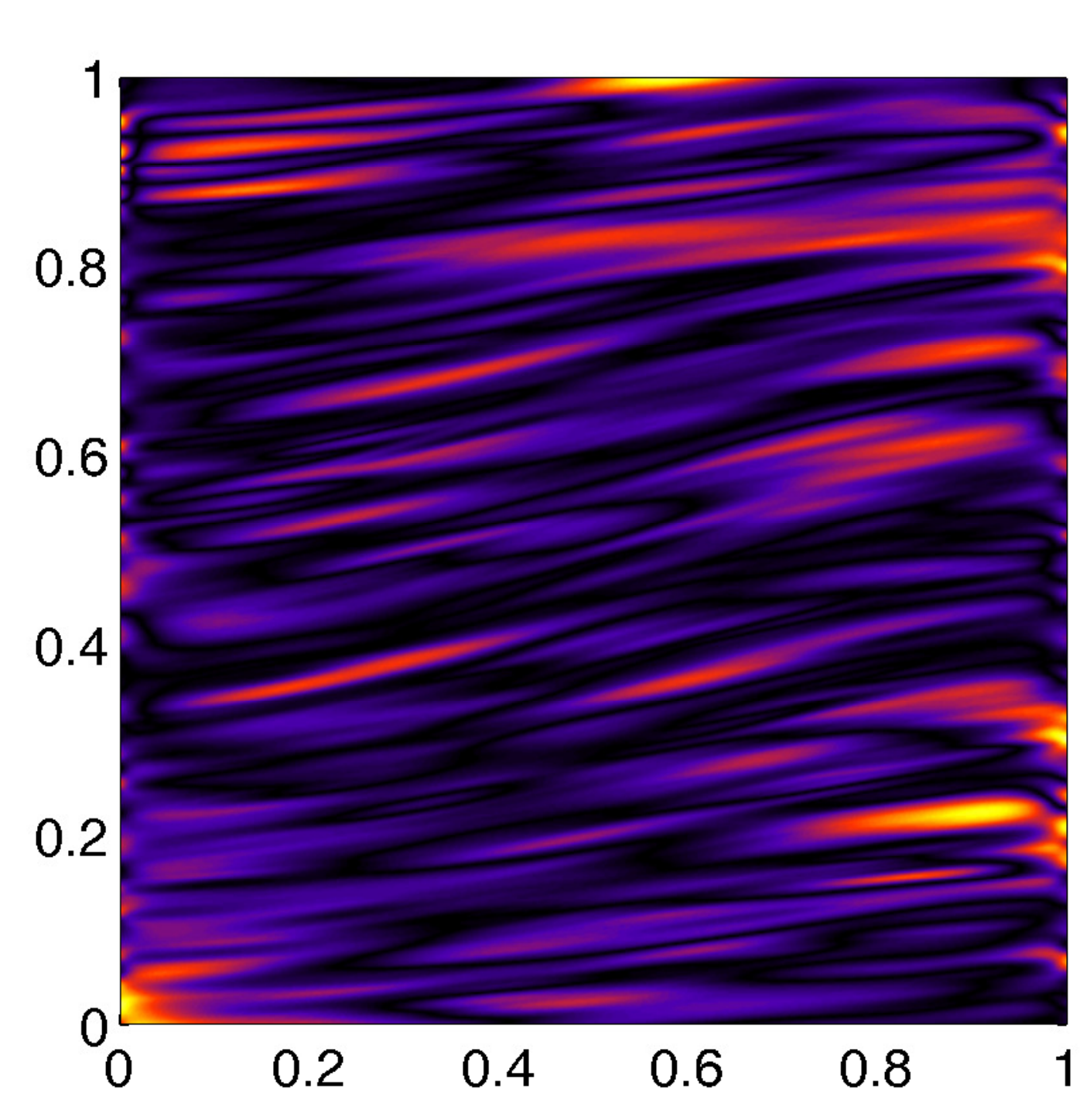}\label{Erm2A}}\hspace{1mm}
\subfloat[$\displaystyle\frac{1}{c}|\mu_\text{ref}(\bx)-\mu^{(2,3)}(\bx)|$]{\includegraphics[height=0.21\textheight]{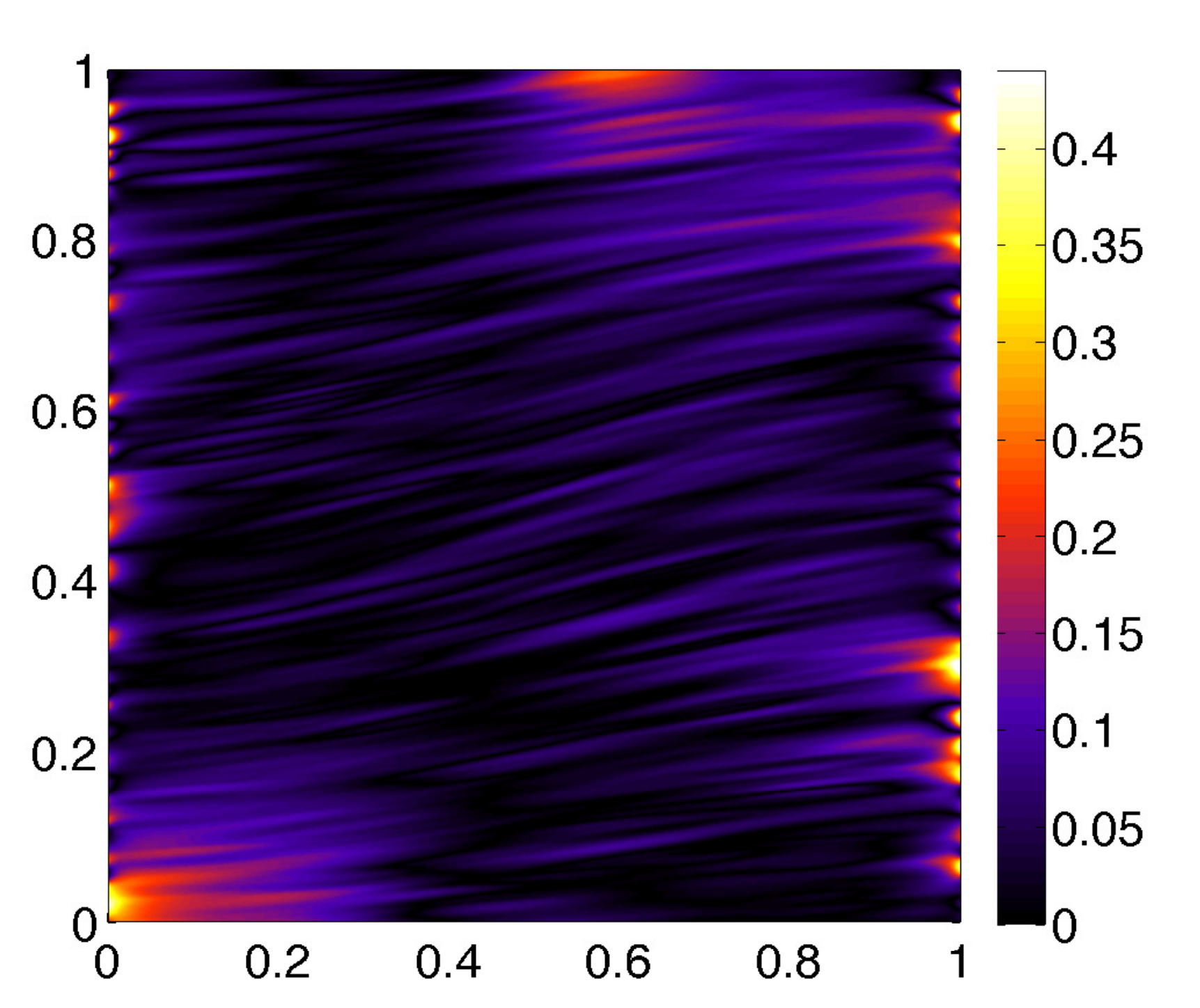}\label{ErmA}}
\caption{\emph{Configuration 1:} Normalized maps of error on the reconstruction relatively to the contrast parameter $c$ with $c=10^{-2}$.}
\end{figure}

\begin{figure}[thb]	
\centering
\subfloat[$\kappa^{(1)}(\bx)-\kappa_0$]{\includegraphics[height=0.21\textheight]{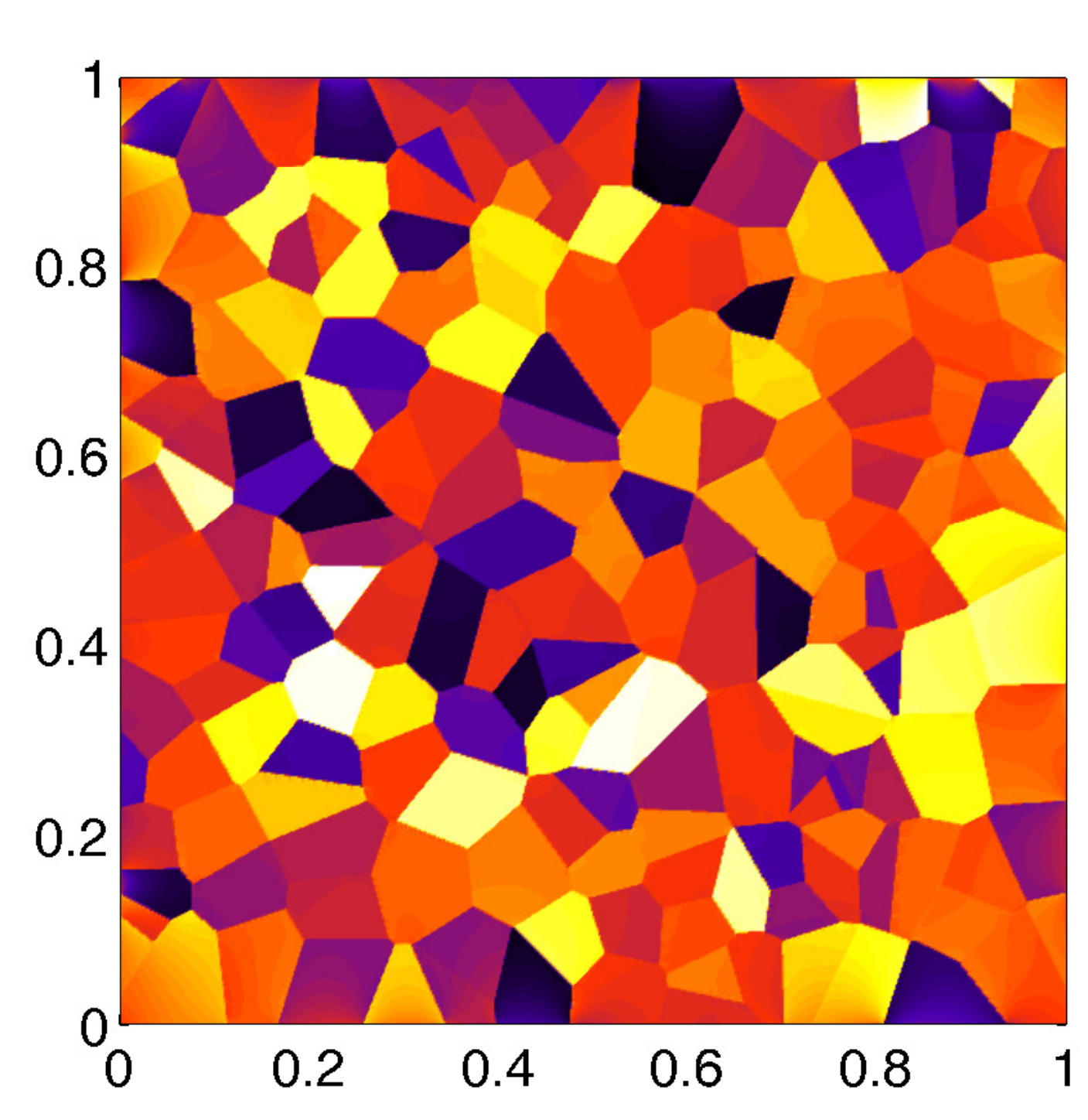}\label{kappaId}}\hspace{1mm}
\subfloat[$\mu^{(2)}(\bx)-\mu_0$]{\includegraphics[height=0.21\textheight]{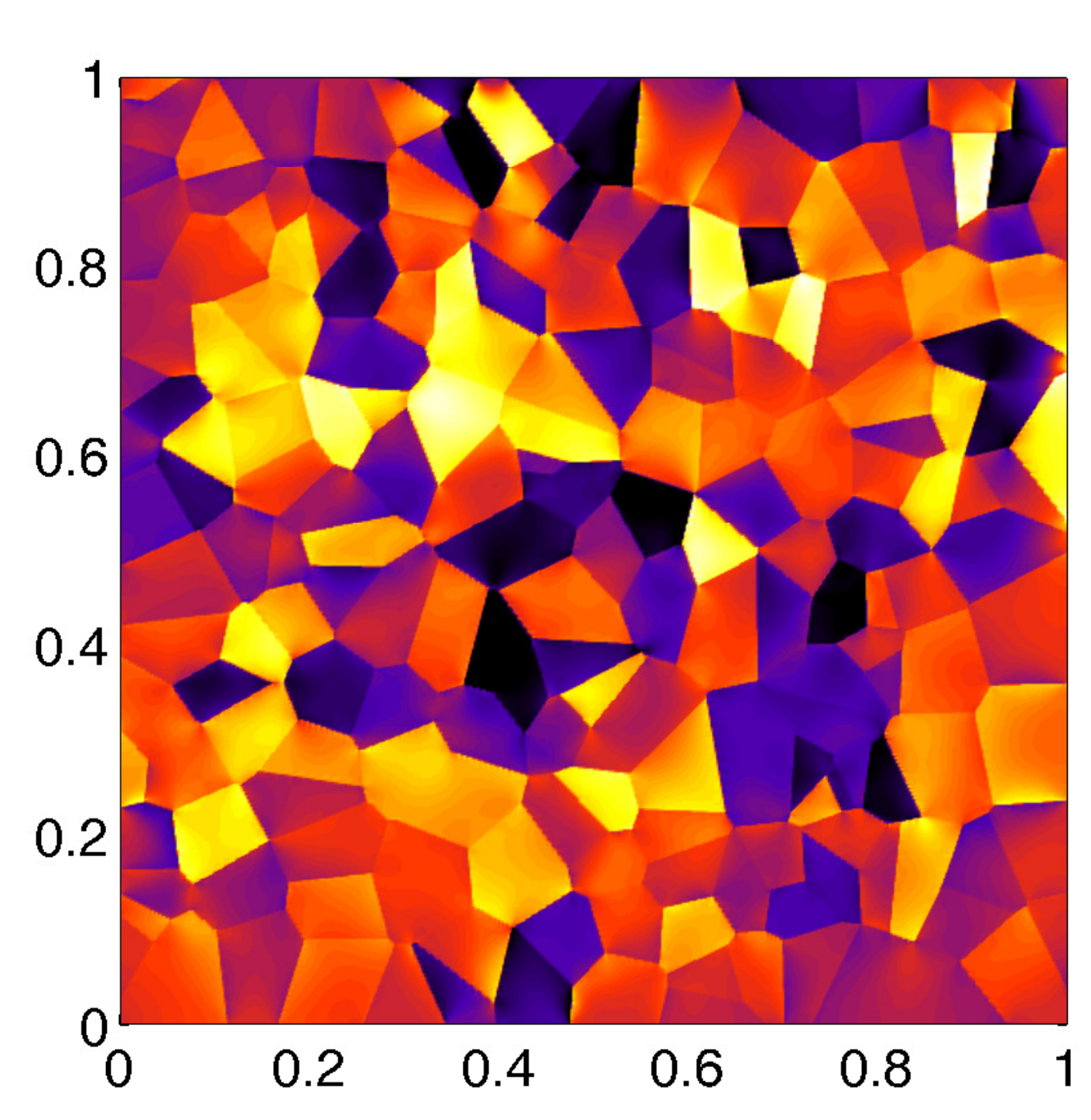}\label{muId2}}\hspace{1mm}
\subfloat[$\mu^{(2,3)}(\bx)-\mu_0$]{\includegraphics[height=0.21\textheight]{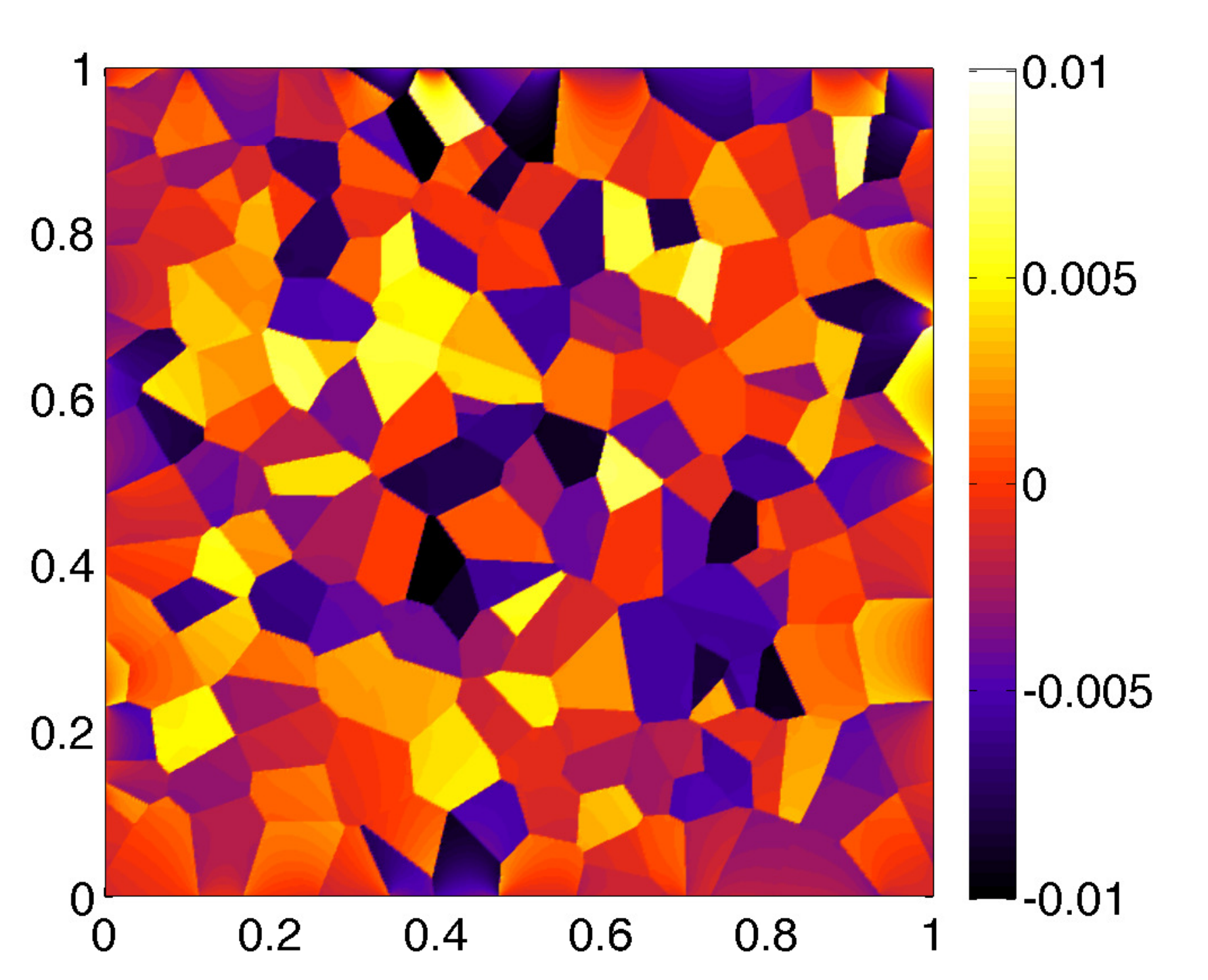}\label{muId}}
\caption{\emph{Configuration 2:} Reconstructed elasticity maps in the case where $c=10^{-2}$.}
\end{figure}

\begin{figure}[htb]	
\centering
\subfloat[$\displaystyle\frac{1}{c}|\kappa_\text{ref}(\bx)-\kappa^{(1)}(\bx)|$]{\includegraphics[height=0.21\textheight]{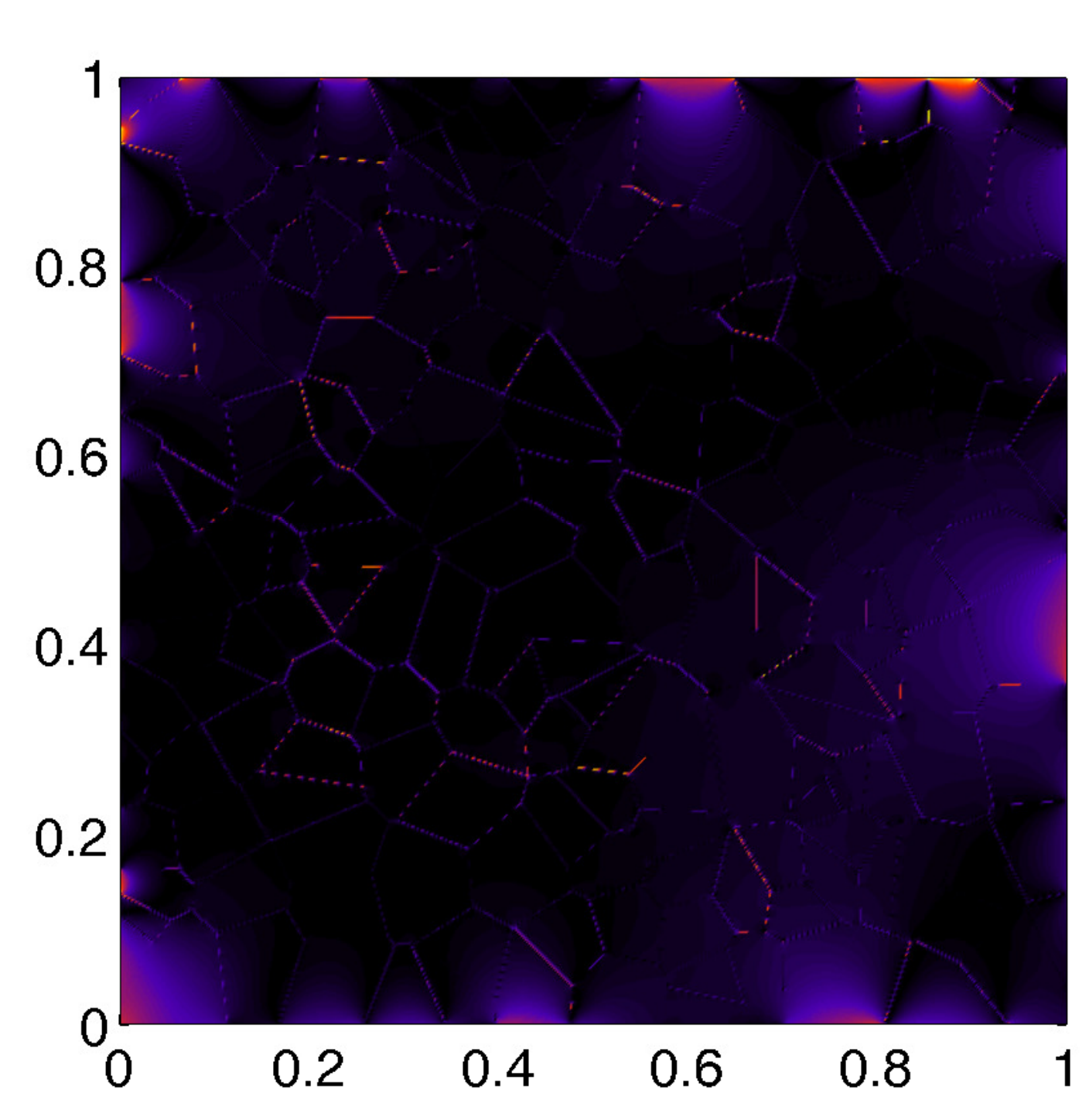}\label{Erk}}\hspace{1mm}
\subfloat[$\displaystyle\frac{1}{c}|\mu_\text{ref}(\bx)-\mu^{(2)}(\bx)|$]{\includegraphics[height=0.21\textheight]{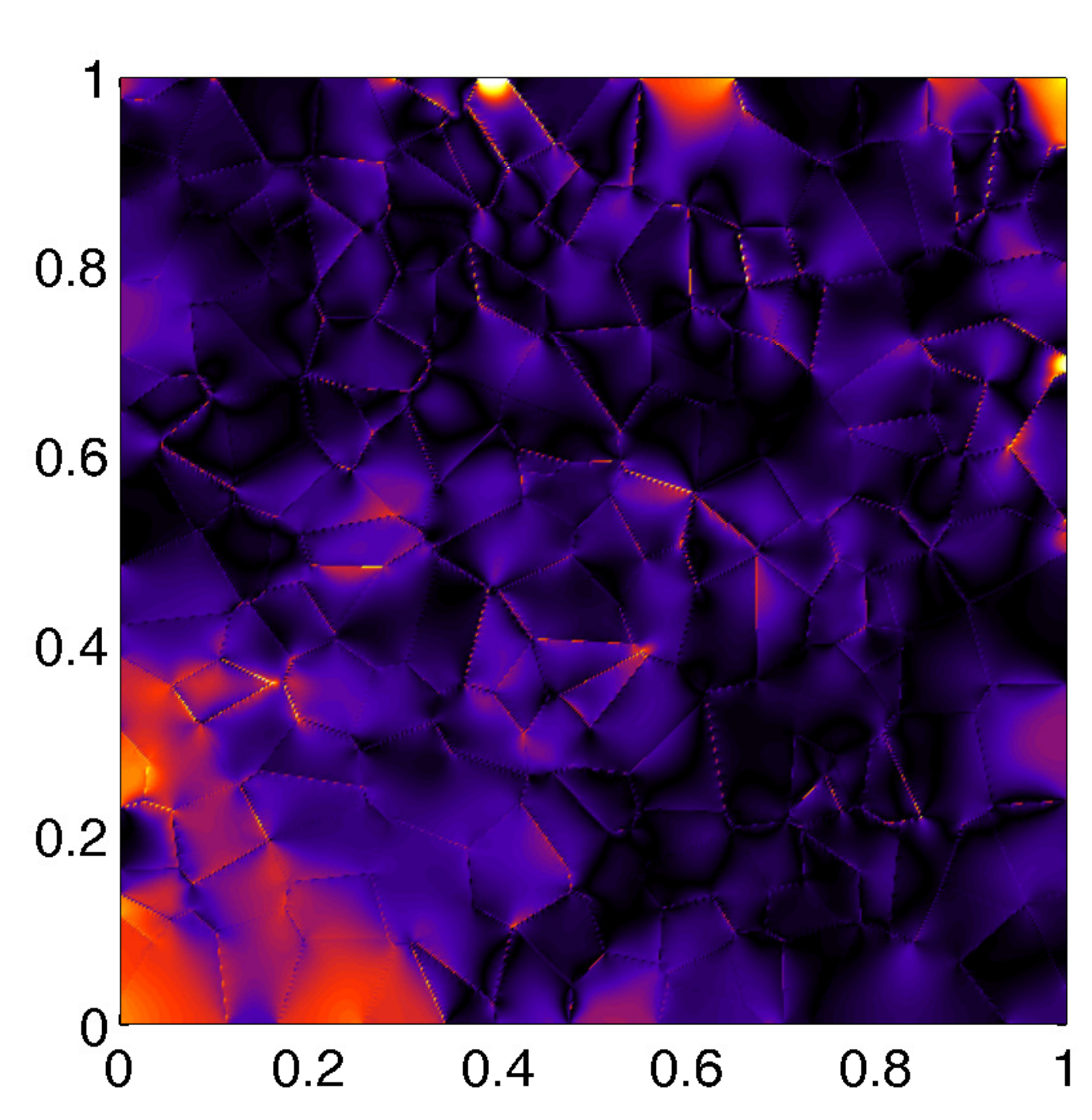}\label{Erm2}}\hspace{1mm}
\subfloat[$\displaystyle\frac{1}{c}|\mu_\text{ref}(\bx)-\mu^{(2,3)}(\bx)|$]{\includegraphics[height=0.21\textheight]{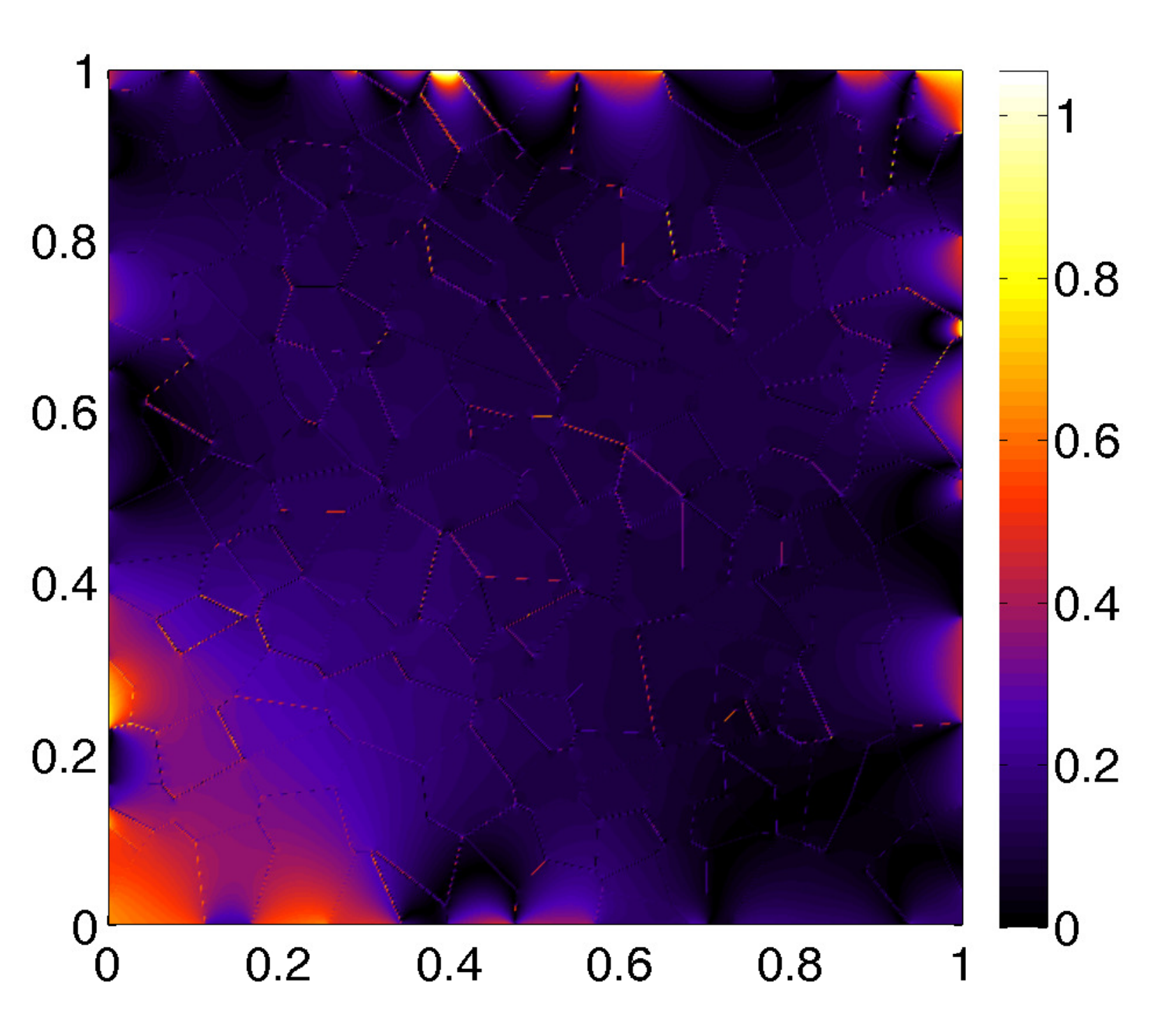}\label{Erm}}
\caption{\emph{Configuration 2:} Normalized maps of error on the reconstruction relatively to the contrast parameter $c$ with $c=10^{-2}$.}
\end{figure}

For the first set of examples, the contrast value is chosen as $c=10^{-2}$. The reconstructed elasticity maps of figures \ref{kappaIdA}, \ref{muIdA} and \ref{kappaId}, \ref{muId} are in good agreements with the reference parameter maps of Fig. \ref{fig:mat:parA} and \ref{fig:mat:par} respectively. The corresponding relative errors (Fig. \ref{ErkA} \ref{ErmA} and \ref{Erk}, \ref{Erm}) are consistent with the first-order identities of propositions \ref{prop:sph:rela} and \ref{prop:dev:rela}. Moreover, as expected from the translation-invariant approximation discussed at Eqn. \eqref{trans:inv:approx}, errors tend to localize at the domain boundary, so that the reconstructions are optimal within the vicinity of the center of the images. Included here for comparison, the reconstruction $\mu^{(2)}$ of the shear modulus in Figs. \ref{muId2A} and \ref{muId2}, and that is incorrectly based on only one deviatoric strain field measurement, yields larger discrepancies with the reference compared to $\mu^{(2,3)}$. The latter is indeed computed using two strain field measurements as required from Proposition \ref{prop:dev:rela}.

\subsection{Influence of contrast amplitude}

Here, we discuss the quality of the reconstructions that are obtained for larger moduli fluctuations. Considering configuration 2, then the relative reconstruction errors on the bulk and shear moduli are depicted on the figures \ref{fig:cont:k} and \ref{fig:cont:m} for contrast parameter values $c=10^{-1}$, $5\cdot10^{-1}$ and $1$. As expected, the accuracy of the first-order based reconstructions degrades as $c$ increases with significant errors observed within the domain when $c=1$. Indeed, neglecting second-order terms in $O\big(\|\delta\Lm\|^2\big)$ yields larger truncation errors. Nevertheless, the reconstruction formulae perform reasonably well for contrast values $c=10^{-1}$ and $c=5\cdot10^{-1}$. This makes the proposed approach appealing in practical applications where measurement sensitivity is critical.

\begin{figure}[thb]	
\centering
\subfloat[$c=10^{-1}$]{\includegraphics[height=0.21\textheight]{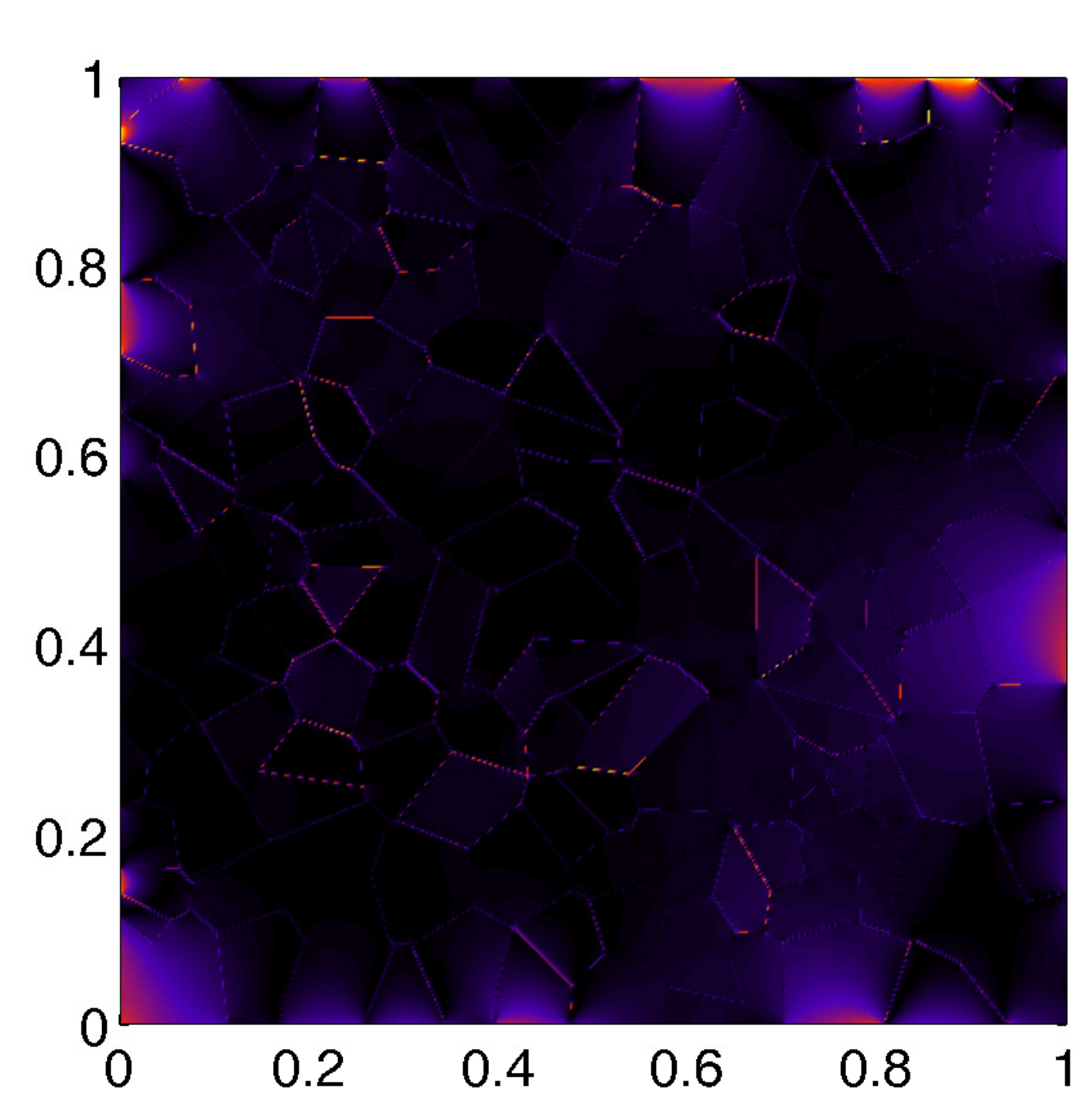}}\hspace{1mm}
\subfloat[$c=5\cdot10^{-1}$]{\includegraphics[height=0.21\textheight]{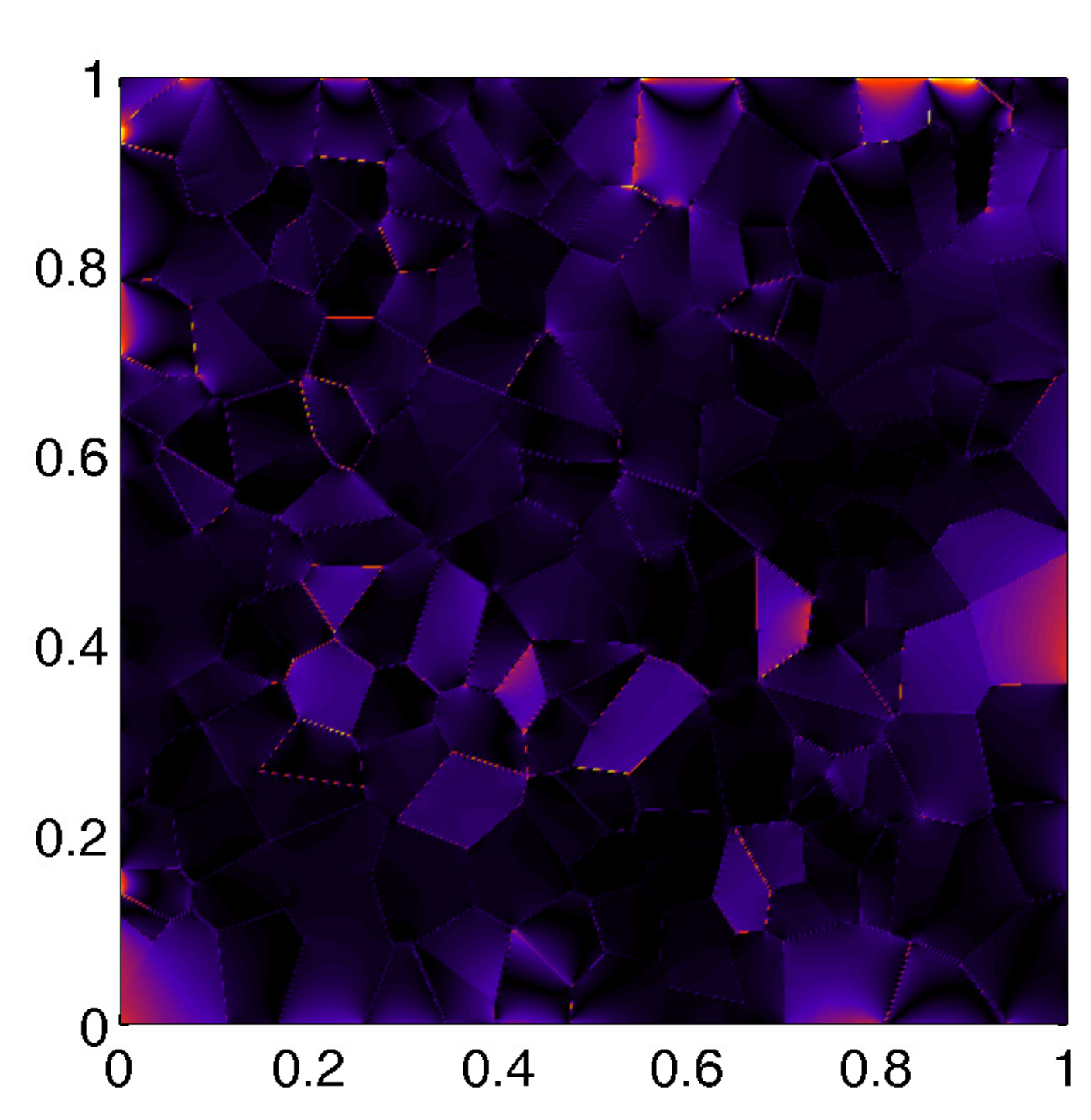}}\hspace{1mm}
\subfloat[$c=1$]{\includegraphics[height=0.21\textheight]{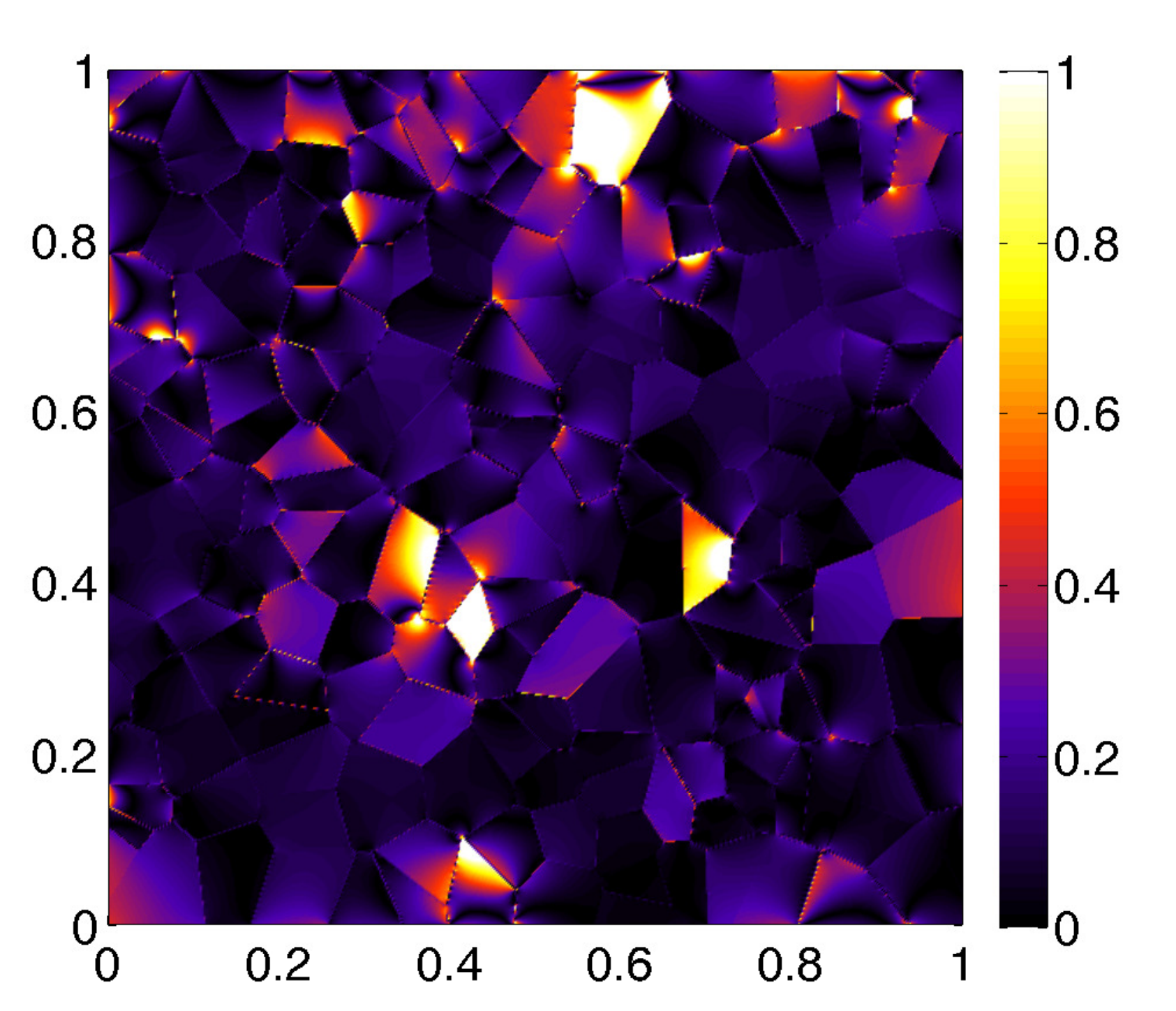}}
\caption{\emph{Configuration 2:} Normalized maps $\displaystyle\frac{1}{c}|\kappa_\text{ref}(\bx)-\kappa^{(1)}(\bx)|$ of error on the reconstruction depending on the contrast value $c$.}
\label{fig:cont:k}
\end{figure}

\begin{figure}[thb]	
\centering
\subfloat[$c=10^{-1}$]{\includegraphics[height=0.21\textheight]{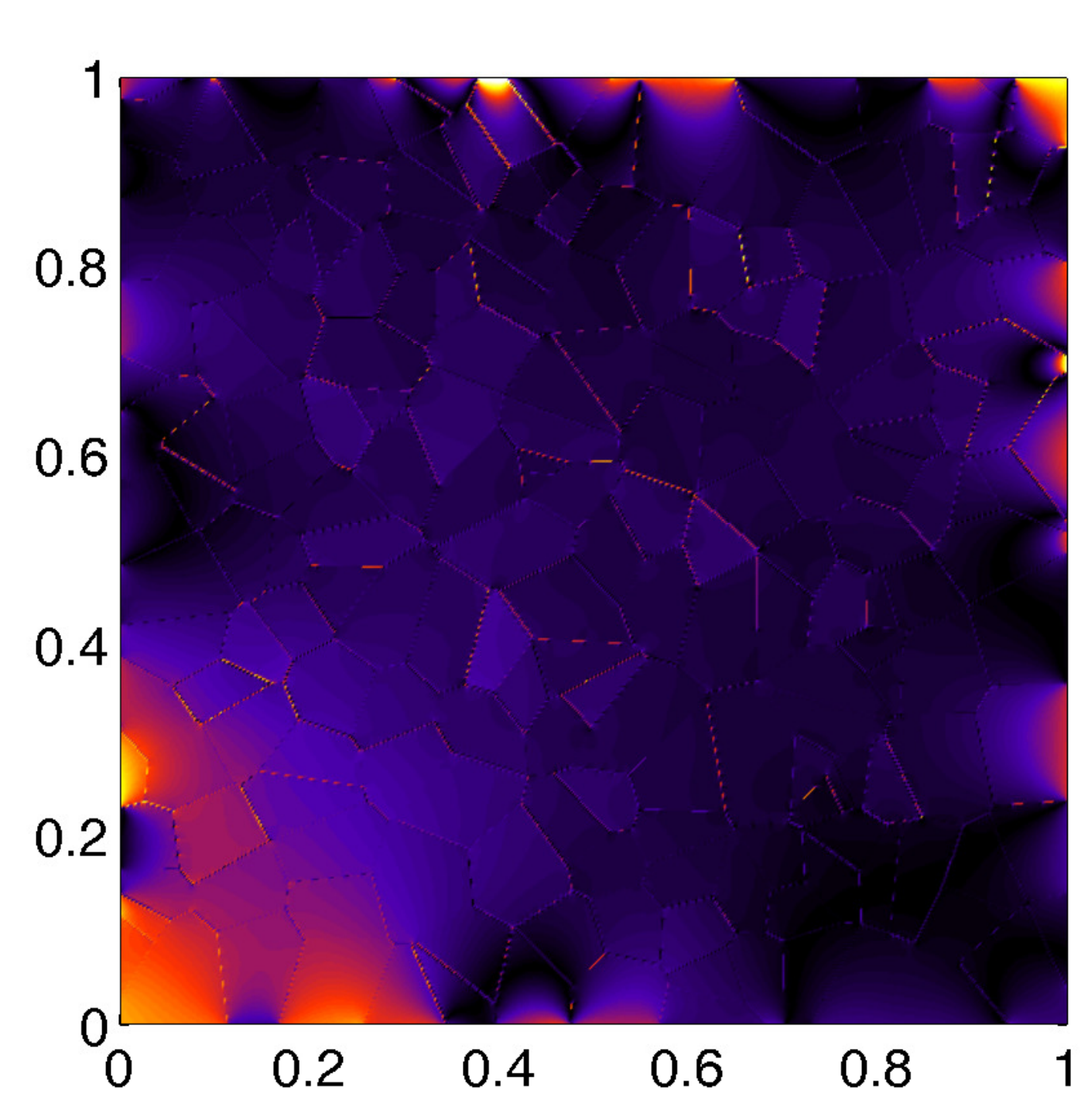}}\hspace{1mm}
\subfloat[$c=5\cdot10^{-1}$]{\includegraphics[height=0.21\textheight]{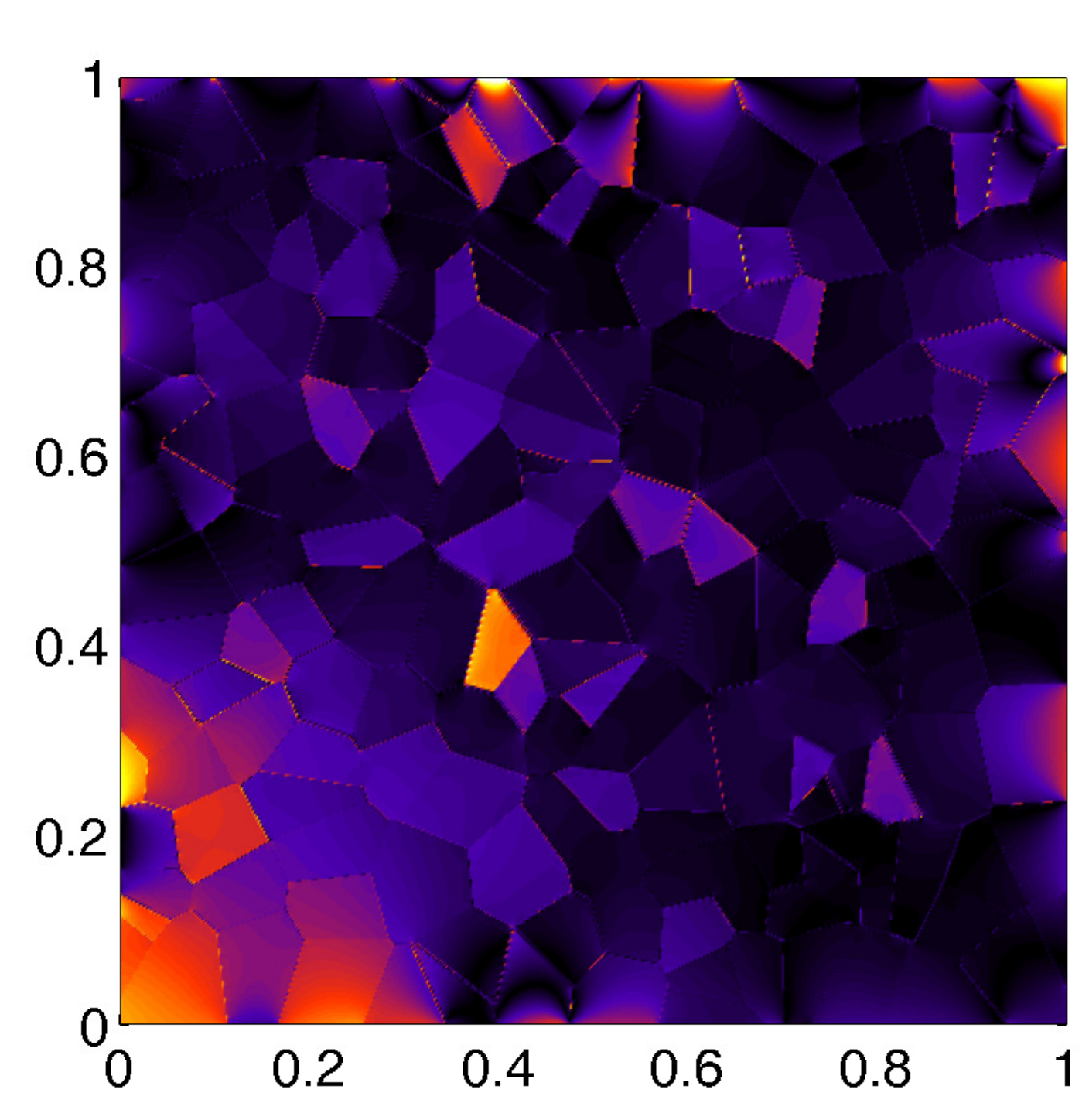}}\hspace{1mm}
\subfloat[$c=1$]{\includegraphics[height=0.21\textheight]{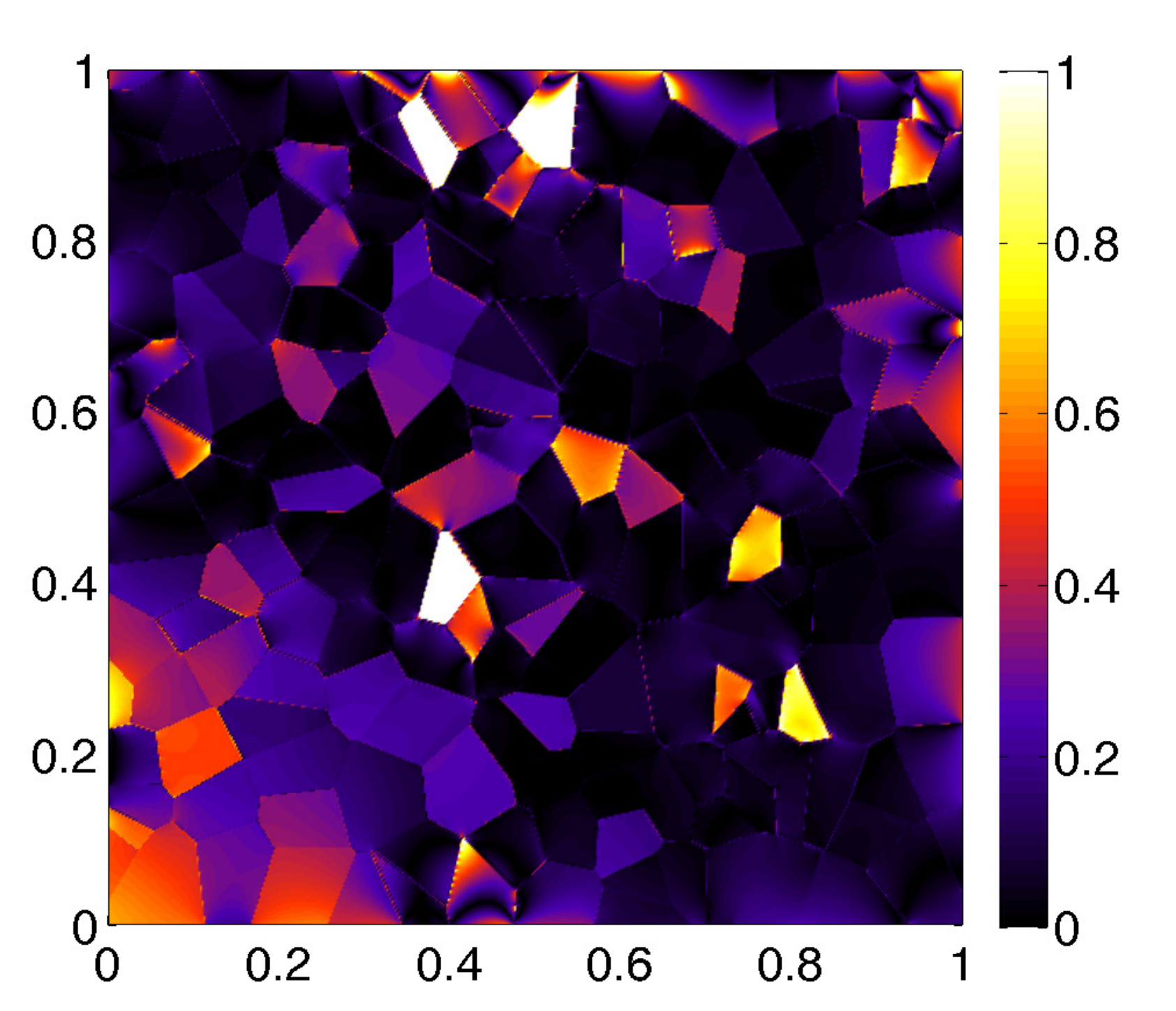}}
\caption{\emph{Configuration 2:} Normalized maps $\displaystyle\frac{1}{c}|\mu_\text{ref}(\bx)-\mu^{(2,3)}(\bx)|$ of error on the reconstruction depending on the contrast value $c$.}
\label{fig:cont:m}
\end{figure}

\section{Conclusions}

For materials with small contrasts, the closed form identities derived in this article allow to convert locally strain field maps into elasticity maps. Based on the integral formulation framework of the periodic elasticity problem which features the periodic Green's operator, then first-order expansions are employed to characterized the behaviors of strain field solutions and of the homogenized elasticity tensor. Then it is shown that the isotropic projection of the Green's tensor is associated with a local contribution to the integral equations considered. This key result then allows to derive local and explicit equations between the bulk or shear modulus and strain field solutions that are respectively associated with purely hydrostatic or deviatoric prescribed macroscopic strains. For a generic material configuration in 2D or 3D then only one strain map is sufficient to reconstruct the bulk modulus while the computation of the shear modulus requires the combination of either 2 or 5 strain field measurements respectively. Furthermore, if the material considered satisfies a property of macroscopic isotropy then it is shown that only one measurement is enough to reconstruct each one of the parameters. The analysis is performed in the periodic case by taking full advantage of a Fourier-based formulation. It is then revisited in the case of bounded media by using a differential operator-based approach. A set of analytical and numerical examples are presented to illustrate the obtained results. In particular the discussed numerical examples suggest that the proposed approach performs reasonably well for moderately small contrast values for which it is less sensitive to data pollution and thus relevant in practical applications. The proposed reconstruction formulae can be readily used, in first approximation, to identify elasticity parameters from full strain field data.

\section*{Acknowledgements}

Pierre Suquet and Herv\'e Moulinec are gratefully acknowledged for fruitful discussions. The Author is thankful to Marc Bonnet for his comments that helped to improve this article.

\appendix

\section{Tensor identities}\label{app:tensor}

The identity tensors on $\sot\mathbb{R}^d$ and $\sots \!\big(\!\sots\! \mathbb{R}^d\big)$ are respectively denoted as $\bm{I}$ with $I_{ij}=\delta_{ij}$ and $\bm{\mathcal{I}}$ with components $\mathcal{I}_{ijk\ell}=\frac12(\delta_{ik}\delta_{j\ell}+\delta_{i\ell}\delta_{jk})$. Moreover, in dimension $d$, one defines the tensors  
\[
\bJ=\frac1d \bm{I}\otimes\bm{I} \quad\text{and}\quad \bK=\bm{\mathcal{I}}-\bJ,
\]
which satisfy $\bJ\dc\bJ=\bJ$, $\bK\dc\bK=\bK$ and $\bJ\dc\bK=\bK\dc\bJ=\bm{0}$ while one introduces the notation
\begin{equation}
\nI=\bm{\mathcal{I}}\qc\bm{\mathcal{I}}=\frac{d(d+1)}{2}, \quad \nJ=\bJ\qc\bJ=1,\quad \nK=\bK\qc\bK=\frac{d(d+1)}{2}-1.
\label{id:qc}
\end{equation}
On noting that for any second-order tensor $\btau\in\sots\mathbb{R}^d$ one has
\[
\bJ\dc\btau=\frac1d(\bm{I}\dc\btau)\bm{I}=\frac1d\tr[\btau]\bm{I} \quad\text{and}\quad\bK\dc\btau=\btau-\frac1d\tr[\btau]\bm{I}=\dev[\btau],
\]
then the tensors $\bJ$ and $\bK$ are identified as the orthogonal projectors on the space of spherical and deviatoric second-order tensors respectively. Moreover they constitute a basis of symmetric and isotropic forth-order tensors \cite{Knowles}: any tensor $\bm{\mathcal{A}}\in \sots \!\big(\!\sots\! \mathbb{R}^d\big)$ isotropic can be written as $\bm{\mathcal{A}}=a\,\bJ+b\,\bK$ where the components $a$, $b$ read
\begin{equation}\label{comp:ela}
a=\frac{1}{\nJ}\bm{\mathcal{A}}\qc\bJ=\frac{1}{\nJ d}\,\mathcal{A}_{iijj}, \qquad b=\frac{1}{\nK}\bm{\mathcal{A}}\qc\bK=\frac{1}{\nK}\big(\mathcal{A}_{ijij}-\frac1d\mathcal{A}_{iijj}\big).
\end{equation}
\\

Considering a macroscopic strain $\bbeps$, then the associated strain field solution $\beps$ satisfies the orthogonal decomposition
\begin{equation}\label{orth:decomp}
\beps(\bx)=\eps\para(\bx)\, \frac{\bbeps}{\| \bbeps \|} + \beps\orth(\bx) \quad \text{with}\quad\eps\para(\bx)=\frac{\beps(\bx):\bbeps}{\| \bbeps \|} \quad \text{and} \quad \eps\orth=\|\beps\orth\|
\end{equation}
where $\|\btau\|^2=\btau:\btau$. Therefore $\eps\para(\bx),\, \eps\orth(\bx)\in\mathbb{R}$ denote the norms of the components of the strain field that are respectively parallel and orthogonal to $\bbeps$. Moreover one has
\begin{equation}\label{pyth:eps}
\|\beps(\bx)\|^2=\eps\para(\bx)^2+\eps\orth(\bx)^2.
\end{equation}

Finally, the so-called hydrostatic strain $\epsz$ and Von Mises equivalent strain $\epse$ are defined as the scalar quantities
\[
\epsz=\frac1d\tr[\beps] \quad\text{and}\quad\epse=\left(\frac{d-1}{d}\,\dev[\beps]\dc\dev[\beps]\right)^{\!\!1/2}.
\]
Note that the results obtained in this article are independent of the coefficients featured in the definitions of $\epsz$ and $\epse$. Using the above tensor projectors, these terms can be rewritten as
\begin{equation}\label{def:eps:z:eq}
\epsz^2=\frac1d \, \bJ\qc(\beps\otimes\beps) \quad \text{and}\quad \epse^2= \frac{d-1}{d}\, \bK\qc(\beps\otimes\beps),
\end{equation}
and they satisfy the identity 
\begin{equation}\label{pyth:JK}
\|\beps(\bx)\|^2=d\,\epsz(\bx)^2+\frac{d}{d-1}\epse(\bx)^2.
\end{equation}

\section{Fourier transform}\label{app:Fourier}

For $f \in L^2_\mathrm{per}(\mathcal{V},\mathbb{R})$, the Fourier transform and its inverse are defined by
\[
 \hat{f}(\bxi)=\mathscr{F}\big[f\big](\bxi)=\frac{1}{|\mathcal{V}|}\int_{\mathcal{V}}f(\bx)e^{-2\pi \mathrm{i}\,\bx\cdot\bxi}\td\bx, \qquad f(\bx)=\mathscr{F}^{-1}\big[\hat{f}\,\big](\bx)=\sum_{\bxi\in\mathcal{L}'} \hat{f}(\bxi)e^{2\pi \mathrm{i}\, \bx\cdot\bxi},
\]
where $\mathcal{L}'$ denotes the reciprocal lattice associated with the lattice $\mathcal{L}$ defined by $\mathcal{V}$.
Given $\mathcal{V}$-periodic functions $f$, $g$ the convolution theorem reads
\begin{equation}\label{conv:thm}\begin{aligned}
& [f\ast g](\bx)=\frac{1}{ |\mathcal{V}| }\int_{\mathcal{V}}f(\by)\, g(\bx-\by)\td\by=\mathscr{F}^{-1}\big[\hat{f}\,\hat{g}\big](\bx),\\
& [\hat{f}\ast \hat{g}](\bxi)= \sum_{\bet\in\mathcal{L}'} \hat{f}(\bet)\,\hat{g}(\bxi-\bet)=\mathscr{F}\big[fg\big](\bxi).
\end{aligned}\end{equation}
Moreover, the combination of space convolution and double inner product will be denoted by $\bm{f}\convd\bm{g}$, the operation $\convd$ being thus defined by replacing the featured product by the tensor inner product ``$:$'' in the above definition. Lastly, one defines for all $\bm{f}$, $\bm{g}$ in $L^2_\mathrm{per}(\mathcal{V},\mathbb{R}^d)$ an inner product as
\begin{equation}\label{def:norm}
\langle \bm{f}\cdot\bm{g} \rangle= \frac{1}{|\mathcal{V}|}\int_\mathcal{V}\bm{f}(\bx)\cdot\bm{g}(\bx)\,\text{d}\bx=\sum_{\bxi\in\mathcal{L}'} \hat{\bm{f}}(\bxi)\cdot{\hat{\bm{g}}}(\bxi)^*,
\end{equation}
with ${\hat{\bm{g}}}(\bxi)^*=\hat{\bm{g}}(-\bxi)$ being the complex conjugate of the real-valued function ${\bm{g}}$ and the second equality in \eqref{def:norm} being due to Plancherel's theorem. This inner product is extended to the spaces $L^2_\mathrm{per}(\mathcal{V}, \mathbb{R})$ and $L^2_\mathrm{per}(\mathcal{V},\sots \mathbb{R}^d)$ by replacing the simply contracted product in \eqref{def:norm} by a product or a doubly contracted product respectively.


\end{document}